\documentclass[11pt]{amsart}
\usepackage[margin=3cm]{geometry}
\usepackage{amssymb, amsmath, amsthm}
\usepackage{verbatim}
\usepackage{color}
\usepackage{bbm}
\usepackage{graphicx, subfig}
\usepackage{enumerate}
\usepackage[sort&compress,numbers]{natbib}
\usepackage{hyperref}
\usepackage{microtype}

\newcommand{\law}       {\mathcal{L}}
\newcommand{\eqlaw} {\stackrel{\law}{=}}
\newcommand{\eqdef} {\stackrel{\textrm{def}}{=}}
\newcommand{\inlaw} {\overset{\law}{\to}}
\newcommand{\inprob}    {\overset{\mathcal{P}}{\to}}

\newcommand{\ex}[1] {\mathbf{E}\left\{ #1 \right\}}
\newcommand{\pr}[1] {\mathbf{P}\left\{ #1 \right\}}
\newcommand{\var}[1]    {\mathbf{Var}\left\{ #1 \right\}}
\newcommand{\event}[1]  {\left[ #1 \right]}
\newcommand{\ind}[1]    {\mathbbm{1}_{#1}}

\newcommand{\leg}[1]    {\Lambda^*\left( #1 \right)}
\renewcommand{\exp}[1]  {\operatorname{exp}\left( #1 \right)}
\newcommand{\Ord}[1]    {O\left( #1 \right)}
\newcommand{\floor}[1]  {\left\lfloor #1 \right\rfloor}
\newcommand{\ceil}[1]   {\left\lceil #1 \right\rceil}

\newcommand{\wt}        {\widetilde}
\newcommand{\eps}       {\varepsilon}
\newcommand{\R}     {\mathbb{R}}
\newcommand{\N}     {\mathbb{N}}
\newcommand{\F}     {\mathcal{F}}

\newcommand{\red}[1]    {{#1}}

\newtheorem{theorem}{Theorem}
\newtheorem{lemma}{Lemma}
\newtheorem{proposition}{Proposition}
\newtheorem{corollary}{Corollary}
\theoremstyle{definition}
\newtheorem*{remark}{Remark}

\newtheorem*{claim}{Claim}

\title{Depth properties of Scaled Attachment \linebreak Random Recursive Trees}
\author{Luc Devroye}
\address{School of Computer Science, McGill University, Montreal, Canada H3A 2K6.}
\email{luc@cs.mcgill.ca}
\author{Omar Fawzi}
\address{School of Computer Science, McGill University, Montreal, Canada H3A 2K6.}
\email{ofawzi@cs.mcgill.ca}
\author{Nicolas Fraiman}
\address{Department of Mathematics and Statistics, McGill University, Montreal, Canada H3A 2K6.}
\email{fraiman@math.mcgill.ca}


\thanks{Research supported by an NSERC
Discovery Grant program. Authors' address: School of Computer Science and Department of Mathematics and Statistics, McGill University, Montreal, Canada H3A 2K6 }

\date{\today}

\begin{document}

\maketitle

\begin{abstract}
We study depth properties of a general class of random recursive
trees where each node $i$ attaches to the random node
$\floor{iX_i}$ and $X_0, \dots, X_n$ is a sequence of i.i.d.\
random variables taking values in $[0,1)$. We call such trees
scaled attachment random recursive trees (\textsc{sarrt}). We
prove that the typical depth $D_n$, the maximum depth (or height)
$H_n$ and the minimum depth $M_n$ of a \textsc{sarrt} are
asymptotically given by $D_n \sim \mu^{-1} \log n$, $H_n \sim
\alpha_{\max} \log n$ and $M_n \sim \alpha_{\min} \log n$ where
$\mu, \alpha_{\max}$ and $\alpha_{\min}$ are constants depending
only on the distribution of $X_0$ whenever $X_0$ has a density. In
particular, this gives a new elementary proof for the height of
uniform random recursive trees $H_n \sim e \log n$ that does not
use branching random walks.
\end{abstract}

\section{Introduction}
A \textit{uniform random recursive tree} (\textsc{urrt}) $T_n$ of
order $n$ is a tree with $n+1$ nodes labeled $\{0, 1, \dots, n\}$
constructed as follows. The root is labeled $0$, and for $1 \leq i
\leq n$, the node labeled $i$ is inserted and chooses a vertex in
$\{0, \dots, i-1\}$ uniformly at random as its parent. The
asymptotic properties of $T_n$ -- the depth of the last inserted
node, the height of the tree, the degree distribution,  the number
of leaves, the profile and so forth -- have been extensively
studied starting from \citet{Moo74,Gas77} and \citet{NR70}.
In particular, \citet{Szy90} showed that the depth $D_n$ of node
$n$ is $(1+o(1))\log n$ with probability going to $1$ and
\citet{Pit94} proved that the height $H_n$ is $(e + o(1)) \log n$
with probability going to $1$. Distance measures in a
\textsc{urrt} were also considered by \citet{Dob96,DF99,MM78,Nei02} 
and \citet{SFH06}. For a survey, see \citet{Drm09} and \citet{SM95}.

A natural generalization of this model introduced by \citet{DL95}
is to let a vertex choose $k>1$ parents uniformly. This
construction defines a random directed acyclic graph
($k$-\textsc{dag}), which was used to model circuits \citet{TX96,
AGM99}.

The uniformity condition was relaxed by \citet{Szy87} by letting
the probabilities of being chosen as a parent depend on the degree
of the parent. When the probability of linking to a node is
proportional to its degree, this gives a random plane-oriented
recursive tree, the typical depth of which was studied by
\citet{Mah92} and the height of which was studied  by
\citet{Pit94}. When $k>1$ parents are chosen for each node, the
popular preferential attachment model of \citet{BA99} is obtained.

Motivated by recent work on distances in random $k$-\textsc{dag}s
(\citet{DJ09}) and on the power of choice in the construction of random
trees (\citet{DKM07, Mah09}), we introduce a generalization of
uniform random recursive trees. In a \emph{scaled attachment
random recursive tree} (\textsc{sarrt}), a node $i$ chooses its
parent to be the node labeled $\floor{iX_i}$ where $X_0, X_1,
\dots, X_n$ is a sequence of independent random variables
distributed as $X$ over $[0,1)$. Note that the choice of the
parent here only depends on the labels of previous nodes and not
on their properties relative to the tree (like the degree, for
example). In particular, if $X$ is uniform on $[0,1)$ we get a
\textsc{urrt}. The distribution $\law(X)$ of $X$ is called the
attachment distribution.

We study properties of the depth (path distance to the root of the
tree) of nodes in a \textsc{sarrt} with a general attachment
distribution. We determine the first-order asymptotics for the
depth $D_n$ of the node labeled $n$, the height $H_n = \max_{1
\leq i \leq n} D_i$ of the tree and the minimum depth $M_n =
\min_{n/2 \leq i \leq n} D_i$. Our result gives a new way of computing the
 height of a \textsc{urrt} that is not based on branching
random walks that were used in previous proofs by \citet{Dev87}
and \citet{Pit94}.

Furthermore, setting $X = \max(U_1, \dots, U_k)$ where $U_1,
\dots, U_k$ are independent random variables with uniform
distribution over $[0,1)$, the depth $D_i$ of node $i$ in a
\textsc{sarrt} with attachment $X$ is the distance given by
following the oldest parent from node $i$ to the root in a random
$k$-\textsc{dag} \citep{DJ09, Mah09}. This problem can be seen as
a ``power of choice'' question: how much can one optimize
properties of the tree when each node is given $k$ choices of
parents? A new node is given $k$ choices of parents, and it
selects the best one according to some criterion. In the setting
of this paper, we study selection criteria that only depend on the
labels or arrival times of the potential parents. Our results
describe the influence of a large class of such selection criteria on
the depth of the last inserted node, the height and the minimum
depth of the tree. This holds for a \textsc{urrt} and for almost
any \textsc{sarrt} as well. Some examples are given in Section
\ref{sec:applications}.

\red{
\noindent \textbf{Outline of the results.}
In Section \ref{sec:depth}, we prove a
concentration result and a central limit theorem for $D_n$ for a very
general class of attachment distributions:
\[
\frac{D_n}{\log n} \inprob \frac{1}{\mu} \qquad\text{and}\qquad
\frac{D_n - \mu^{-1}\log n}{ \sigma \sqrt{\mu^{-3}\log n} } \inlaw \mathcal{N}(0,1),
\]
where $\mu$ and $\sigma^2$ are simply the expected value and the
variance of $-\log X$, $\mathcal{N}(0,1)$ denotes the standard
Gaussian distribution and the symbols $\inprob$ and $\inlaw$ refer
to convergence in probability and convergence in distribution.
This generalizes a result of \citet{Mah09}. In Sections
\ref{sec:height} and \ref{sec:mindepth}, we prove the main theorems 
(Theorems \ref{thm:heightmain} and \ref{thm:mindepthmain}) of this 
paper: if $\law(X)$ has a density on $[0,1)$, then there exist constants
$\alpha_{\max}$ and $\alpha_{\min}$ such that
\[
\lim_{n\to \infty} \frac{H_n}{\log n} = \alpha_{\max}
\quad\text{almost surely},
\qquad\text{and}\qquad
\frac{M_n}{\log n} \inprob \alpha_{\min},
\]
where $H_n$ and $M_n$ denote the height and minimum depth of the
\textsc{sarrt} with attachment $X$. These constants are defined as
the solutions of equations involving a rate function associated with
$\log X$. The proof of these results uses a second moment method.
\red{ 
The main difficulty in the proof is in controlling the dependencies between 
the paths up to the root that originate from different nodes. We also prove
 that $\lim_{n \to \infty} \frac{\ex{H_n}}{\log n} = \alpha_{\max}$.
} 

The different results are applied to study the properties of
various path lengths in a random $k$-\textsc{dag} in Section
\ref{sec:applications}. Lastly, we include an appendix proving
some simple properties of the large deviation rate functions used.

\noindent \textbf{Notation.}
As introduced earlier, the symbols $\inprob$ and $\inlaw$ refer to convergence 
in probability and convergence in distribution respectively. For random 
variables $X$ and $Y$, we write $\law(X)$ for the distribution of $X$ and 
$X \eqlaw Y$ when $X$ and $Y$ have the same distribution. For a general 
random variable $X \in [0,1)$, we define
\[
\mu = \ex{-\log X} \geq 0 \qquad\text{and}\qquad \sigma^2 = \var{-\log X}.
\]
If $X$ has an atom at $0$, then $\mu = \sigma = + \infty$.
If $\mu = + \infty$, then we define $\sigma = + \infty$.
A \textsc{sarrt} with attachment distribution $\law(X)$ is described by a sequence $X_0, X_1, \dots, X_n$ of i.i.d.\
random variables distributed as $X$. The parent of node $i$ is labeled $\floor{iX_i}$. The root of
the tree is labeled $0$ and $L(n, j)$ is the (random) label of the
$j$-th grandparent of $n$ on its path to the root. Note that $L(n,
j+1) = \floor{L(n,j) X_{L(n,j)}}$ and that $L(n,0)=n$. 
The depth $D_i$ of node $i$ is defined by
$
D_i=\min\{j \geq 0: L(i,j) = 0\}.
$
}

\section{The depth of a typical node}\label{sec:depth} 
We look at the sequence of labels from node $n$
to the root as a renewal process. We have
\begin{align*}
D_n&=\min\{j \geq 0: L(n,j) = 0\} \\
&=\min\{j \geq 0: \floor{\dots\floor{\floor{nX_n}X_{L(n,1)}}\dots X_{L(n,j-1)}} =0\}.
\end{align*}
Note that
\[
nX_nX_{L(n,1)}\dots X_{L(n,j-1)} - j
\leq \floor{\dots\floor{\floor{nX_n}X_{L(n,1)}}\dots X_{L(n,j-1)}}
\leq nX_nX_{L(n,1)}\dots X_{L(n,j-1)}.
\]

\begin{remark}
Since $X\in [0,1)$, we have $\mu = \ex{-\log X} > 0$.
Thus, the following theorem covers all the possible cases.
\end{remark}

\begin{theorem}\label{thm:depth}
\begin{flalign*}
&\text{If $\mu = +\infty$, then}\quad \frac{D_n}{\log n} \inprob 0 \quad \text{ and }
\quad \lim_{n \to \infty} \frac{\ex{D_n}}{\log n} = 0. & \tag{A}\label{eq:deptha} \\
&\text{If $\mu < +\infty$, then}\quad \frac{D_n}{\log n} \inprob \frac{1}{\mu} \quad \text{ and }
\quad \lim_{n \to \infty} \frac{\ex{D_n}}{\log n} = \frac{1}{\mu}. & \tag{B}\label{eq:depthb} \\
&\text{If $\mu < +\infty$ and $0 < \sigma^2 < +\infty$, then}\quad
    \frac{D_n -\log n/\mu}{\sigma\sqrt{\log n/\mu^3}} \inlaw \mathcal{N}(0,1).
    & \tag{C}\label{eq:depthc} \\
&\text{If $\mu < +\infty$ and $\sigma^2 = 0$, then}\quad
    D_n -\log n/\mu = o\left(\sqrt{\log n}\right) \quad\text{almost surely.}
    & \tag{D}\label{eq:depthd}
\end{flalign*}
\end{theorem}

\begin{remark}
\citet{Mah09} proved a similar result using generating functions
for the case $X \eqlaw \max(U_1, \dots, U_k)$ and $X \eqlaw
\min(U_1, \dots, U_k)$. Details are given in Section
\ref{sec:applications}.
\end{remark}

\begin{proof}
We consider an auxiliary renewal process $R_t=\sup \big\{j:
\sum_{i=1}^j Z_i \leq t \big\}$ with interarrival times
distributed as $Z_i\eqlaw -\log X$ for all $i$. When $\mu < +\infty$,
the strong law of large numbers for renewal processes gives that
$R_t/t \to 1/\mu$ almost surely (see \citealp{Ros96}, Proposition
3.3.1). Moreover, the elementary renewal theorem implies that
$\ex{R_t}/t \to 1/\mu$. The following claim handles the case $\mu 
= +\infty$.

\begin{claim}
For $\mu = +\infty$, $\lim_{t \to \infty} \frac{R_t}{t} = 0$ with
probability $1$ and $\lim_{t \to \infty} \frac{\ex{R_t}}{t} = 0$.
\end{claim}
\begin{proof}
For fixed $b > 0$, let $\wt{Z}_i = \min(Z_i, a)$ where $a$ is chosen 
so that $\mathbf{E}\big\{\wt{Z}_i\big\}\geq b$. Consider the renewal
process $\wt{R}_t$ with interarrival times $\wt{Z}_i$. By the fact
that $R_t \leq\wt{R}_t$  and the law of large numbers for
$\wt{R}_t$ we have, for sufficiently large $t$, $R_t/t \leq
\wt{R}_t/t < 2/b$ almost surely. Since $b$ is arbitrary, we have
$R_t/t\to 0$ with probability 1. The convergence of the expected
value is proved in a similar way. This concludes the proof of the
claim.
\end{proof}

We upper bound the depth of node $n$ by
\begin{align*}
D_n &\leq \min \left\{ j: nX_nX_{L(n,1)}\dots X_{L(n,j-1)} < 1 \right\} \\
&= \min \left\{ j: \textstyle{\sum}_{i=0}^{j-1} -\log X_{L(n,i)} > \log n \right\}
\eqdef \widehat{D}_n.
\end{align*}
For $n \geq 1$, $\widehat{D}_n \eqlaw R_{\log n} + 1$. So, we have for any $\eps > 0$ that
\begin{equation}\label{eq:dnupper}
\pr{\frac{D_n}{\log n} > \frac{1}{\mu}+\eps} \leq
\pr{\frac{\widehat{D}_n}{\log n} > \frac{1}{\mu}+\eps} =
\pr{\frac{R_{\log n} + 1}{\log n} > \frac{1}{\mu}+\eps} = o(1).
\end{equation}
Since $D_n > 0$, equation \eqref{eq:dnupper} proves part
\eqref{eq:deptha} of the theorem (by writing $1/\mu=0$ when
$\mu=+\infty$).

Similarly, a lower bound is given by
\begin{align*}
D_n &\geq \min \left\{ j: nX_n\dots X_{L(n,j-1)}-j < 1 \right\} \\
&\geq \min \left\{ j: \textstyle{\sum}_{i=0}^{j-1} -\log X_{L(n,i)} > \log n-\log j \right\}.
\end{align*}
Let $j(n) = \floor{\log^2 n}$ and define the event
\[
E_n=\left[\;\sum_{i=0}^{j(n)-1} -\log X_{L(n,i)} > \log n \, \right].
\]
Using the upper bound \eqref{eq:dnupper}, we have that
$\pr{E_n}\to 1$. Also, we have $\log j \leq 2 \log \log n$ and if
we define $f(n)=\log n - 2 \log\log n$, then when $E_n$ holds
\begin{align*}
D_n &\geq \min \left\{j: \textstyle{\sum}_{i=0}^{j-1} -\log X_{L(n,i)} > f(n)\right\} \eqdef \overline{D}_n.
\end{align*}
We have $\overline{D}_n \eqlaw R_{f(n)} + 1$ for $n\geq 2$, and thus,
\begin{equation}\label{eq:dnlower}
\pr{\frac{\overline{D}_n}{\log n} < \frac{1}{\mu}-\eps} =
\pr{\frac{R_{f(n)}+1}{f(n)} \cdot \frac{f(n)}{\log n} < \frac{1}{\mu}-\eps } = o(1),
\end{equation}
by the law of large numbers for renewal processes and the fact that
\[
\lim_{n \to \infty}\, \frac{f(n)}{\log n} = 1.
\]
Combining \eqref{eq:dnupper} and \eqref{eq:dnlower} with the fact
that $\pr{D_n\geq\overline{D}_n}\geq \pr{E_n}$ we obtain
convergence in probability of part \eqref{eq:depthb} of the
theorem. As for the expected value, we have for any $\eps > 0$,
\[
(1/\mu-\eps) \log n \cdot \pr{D_n \geq (1/\mu-\eps) \log n} \leq \ex{D_n} \leq \ex{\widehat{D}_n}
\]
which completes the proof of \eqref{eq:depthb}.

By similar arguments using the central limit theorem for renewal
processes (see \citealp{Ros96}, Theorem 3.3.5) we can prove part
\eqref{eq:depthc} for $D_n$, by showing that
\[
\lim_{n \to \infty} \pr{\frac{\widehat{D}_n-\log n/\mu}{\sigma\sqrt{\log n/\mu^3}} \leq c } = \Phi(c)
\qquad\text{and}\qquad
\lim_{n \to \infty} \pr{\frac{\overline{D}_n-\log n/\mu}{\sigma\sqrt{\log n/\mu^3}} \leq c } = \Phi(c),
\]
where $\Phi$ is the cumulative distribution function of a standard
$\mathcal{N}(0,1)$ variable. The result follows from the fact that
$\overline{D}_n\leq D_n \leq \widehat{D}_n$ with probability going
to $1$ as $n\to\infty$. The first limit is clear and to show the
second limit, write
\[
\frac{\overline{D}_n-\log n/\mu}{\sigma\sqrt{\log n/\mu^3}}
=   \frac{\big( \overline{D}_n-f(n)/\mu \big)+\big( f(n)/\mu - \log n/\mu \big)}{\sigma\sqrt{f(n)/\mu^3}}
    \cdot \sqrt{\frac{f(n)}{\log n}}
\]
where we have 
\[
\lim_{n\to\infty} \frac{f(n)/\mu - \log n/\mu}{\sigma\sqrt{f(n)/\mu^3}} =
\lim_{n\to\infty} \frac{-2\log\log n/\mu}{\sigma\sqrt{f(n)/\mu^3}} =
\lim_{n\to\infty} \frac{-2\log\log n/\mu}{\sigma\sqrt{\log n - 2\log\log n/\mu^3}} = 0.
\]
Also, the central limit theorem for renewal processes implies that
\[
\lim_{n \to \infty} \pr{\frac{\overline{D}_n-f(n)/\mu}{\sigma\sqrt{f(n)/\mu^3}} \leq c } = \Phi(c).
\]

When $\sigma^2=0$, $X=e^{-\mu}\in (0,1)$ almost surely. Then the
label of node $i$ parent is $\floor{ie^{-\mu}}$ and
$L(n,j)=\floor{\floor{\floor{ne^{-\mu}}e^{-\mu}}\dots e^{-\mu}}$
($j$ times) almost surely. Since $ne^{-j\mu} -j \leq L(n,j) \leq
ne^{-j\mu}$ and for $n \geq n_0(\mu)$ we have $ne^{-j\mu} < 1$
when $j > \log n/\mu$ and  $ne^{-j\mu} -j > 1$ when $j < \log
n/\mu$. Then, we have that $|D_n-\log n/\mu|\leq 1$ for $n \geq
n_0$. Therefore we get part \eqref{eq:depthd} of the theorem.
\end{proof}


\section{The height of the tree}\label{sec:height}

\red{
We turn our attention to the height $H_n = \max_{1 \leq i \leq n} D_i$ of a \textsc{sarrt}.  
For a random variable $Y$, we define its cumulant generating function $\Lambda_Y$
and its convex (Fenchel--Legendre) dual $\Lambda^*_Y$ as follows:
\begin{equation}\label{eq:lambdaY}
\Lambda_Y(\lambda) = \log \ex{e^{\lambda Y}} \quad\text{and}\quad
\Lambda^*_Y(z) = \sup_{\lambda\in\R} \big\{\lambda z - \Lambda_Y(\lambda) \big\}.
\end{equation}
Since we mostly use these functions for $Y = \log X$, we omit the subscript in this case.
We write
\begin{equation}\label{eq:lambda}
\Lambda(\lambda) = \log \ex{e^{\lambda \log X}} = \log \ex{X^\lambda} \quad\text{and}\quad
\Lambda^*(z) = \sup_{\lambda\in\R} \big\{\lambda z - \Lambda(\lambda) \big\}
\end{equation}
for the cumulant generating function of $\log X$ and its dual. It
is well known that
$\Lambda^*(z) = \sup_{\lambda \geq 0} \big\{\lambda z - \Lambda(\lambda) \big\}$ for $z \geq \ex{\log X}$ and
$\Lambda^*(z) = \sup_{\lambda \leq 0} \big\{\lambda z - \Lambda(\lambda) \big\}$ for $z \leq \ex{\log X}$.
This is proved along with many properties of $\Lambda^*$ used in
the paper in Appendix \ref{sec:applambda}. We also define
\begin{equation}\label{eq:psi}
\Psi(c)=c\leg{-1/c}
\end{equation}
and
\begin{equation}\label{eq:alphamax}
\alpha_{\max} = \inf\left\{c:\; c > \frac{1}{\mu} \;\text{ and }\; \Psi(c) > 1\right\}
\end{equation}
where we define $1/\mu = 0$ when $\mu = + \infty$.
Proposition \ref{prop:psi} in the appendix shows that the set $\left\{c:\; c >
\frac{1}{\mu} \;\text{ and }\; \Psi(c) > 1\right\}$ is non-empty,
$\alpha_{\max} < +\infty$ and if $X$ is not a constant,
$\alpha_{\max} > 1/\mu$.

The following theorem sums up the results we prove in this section.
\begin{theorem}\label{thm:heightmain}
The height $H_n$ of a \textsc{sarrt} with attachment $X$ having a density satisfies
\[
\lim_{n \to \infty} \frac{H_n}{\log n} = \alpha_{\max} \quad \text{with probability }1,
\qquad\text{and}\qquad \lim_{n \to \infty} \frac{\ex{H_n}}{\log n} = \alpha_{\max},
\]
where $\alpha_{\max}$ is defined in equation \eqref{eq:alphamax}.
\end{theorem}
\begin{remark}
It is worth observing that if $X$ is not constant and $\mu = +\infty$, then $D_n = o(\log n)$ in probability as shown in Theorem \ref{thm:depth}, whereas $H_n = \Theta(\log n)$ in probability. If $X = \alpha \in (0,1)$ with probability $1$, then $\alpha_{\max} = 1/\mu = - 1/\log \alpha$ and it is easy to see that the results of the theorem also hold in this case.
\end{remark}

We start by proving convergence in probability of $\frac{H_n}{\log n}$ in
Sections \ref{sec:heightup} and \ref{sec:heightlow} in the case of
a bounded density. Section \ref{sec:heightup} gives an upper bound
for $\frac{H_n}{\log n}$ with no condition on $X$. The lower bound we
present in Section \ref{sec:heightlow} is more involved and uses
an upper bound on the density in order to bound the dependence
between different paths. In Section \ref{sec:unbounded}, we show
that the lower bound still holds if $X$ has an unbounded density.
Finally, Section \ref{sec:convergence} is devoted to proving
almost sure convergence and convergence in mean as stated in the above theorem.}



\subsection{The height of the tree: upper bound}\label{sec:heightup}
Based on the bounding techniques of \citet{Che52} and \citet{Hoe63}
we can prove the following result.
\begin{lemma}\label{lem:upper}
For any $c>\alpha_{\max}$, we have $\pr{H_n \geq c \log n} \to 0.$
\end{lemma}
\begin{proof}
To simplify the notation, we prove  $\pr{H_n \geq c \log n + 2}
\to 0$ for all $c > \alpha_{\max}$, which is an equivalent
statement. For $t \geq 1$,  applying Markov's inequality, we get
\begin{align*}
\pr{D_n > t} &\leq \pr{ nX_{n} \dots X_{L(n,t-1)} \geq 1 } \\
    &\leq \inf_{\lambda \geq 0} n^\lambda \ex{ X_n^\lambda \dots X_{L(n,t-1)}^\lambda } \\
    &= \inf_{\lambda \geq 0} n^{\lambda} \ex{X^\lambda}^t \\
    &= \inf_{\lambda \geq 0} \exp{ \lambda\log n + \Lambda(\lambda)t \Big.}
\end{align*}
Setting $t=\ceil{c\log n}$, we obtain
\begin{align}\label{eq:npsi}
\pr{D_n \geq c \log n + 2}
    &\leq \inf_{\lambda \geq 0} \exp{ \lambda\log n + \Lambda(\lambda) c \log n  \Big.},
        &\text{(as $\Lambda(\lambda) \leq 0$)} \notag \\
    &\leq \exp{ -\sup_{\lambda \geq 0} \left\{-\frac{\lambda}{c}  - \Lambda(\lambda)\right\} c\log n} \notag \\
    &= \exp{ -c \leg{-1/c} \log n } \notag \\
    &= n^{-\Psi(c)}.
\end{align}

When $\Psi(c) > 0$, the bound in \eqref{eq:npsi} goes to $0$.
Recalling that $c > \alpha_{\max}$ and the definition of
$\alpha_{\max}$ (equation \eqref{eq:alphamax}), we obtain $\Psi(c)
> 1$. Applying a union bound, we get
\begin{align}\label{eq:probheightup}
\pr{H_n > t} = \pr{ \max_{1 \leq i \leq n} D_i > t }
    &\leq \sum_{i=1}^n \pr{ D_i > t } \\
    &\leq  n \pr{ D_n > t } \notag \\
    &\leq n^{1-\Psi(c)} \to 0 \notag
\end{align}
as $n \to \infty$. Note that the last inequality holds because
$\floor{\dots\floor{\floor{iX_i}X_{L(i,1)}}\dots X_{L(i,t-1)}}$ is
stochastically smaller than
$\floor{\dots\floor{\floor{nX_n}X_{L(n,1)}}\dots X_{L(n,t-1)}}$
for $i \leq n$ as the sequence $(X_i)$ is i.i.d.
\end{proof}

In the next section we prove a lower bound on the height of the
tree. We show that for any $c < \alpha_{\max}$, there exists a
node of depth larger than $c \log n$.


\subsection{The height of the tree: lower bound}\label{sec:heightlow}\ \\

\noindent \textbf{Overview of the proof.} It is worth observing first that
the upper bound (Lemma \ref{lem:upper}) does not take into account
the structure of the tree in any way. Introduce the events
$A_x=\event{D_x \geq (\alpha_{\max}-\eps)\log n}$ where $\eps \in
(0, \alpha_{\max})$. We omit the dependence in $\eps$ in this
overview. Applying a second moment inequality sometimes called the
Chung-Erd\H{o}s inequality \citep{CE52}, we get
\begin{equation}\label{eq:chungerdos}
\pr{\bigcup_{x=1}^n A_x}
    \geq \frac{\big( \sum_{x=1}^n \pr{A_x}\big)^2 }
    {\sum_{x\neq y} \pr{A_x \cap A_x} + \sum_{x=1}^n \pr{A_x} }.
\end{equation}
It is not hard to show that $\sum_{x=1}^n \pr{A_x}\to +\infty$ as $n \to \infty.$
Hence, showing that 
\[
\sum_{x \neq y} \pr{A_x \cap A_y} \sim \sum_{x \neq y} \pr{A_x} \pr{A_y} 
\]
would imply that the right hand side of \eqref{eq:chungerdos} goes to $1$. 
This would prove the lower bound on the height that we seek. Therefore, 
our objective is to prove that the collisions between branches of the tree 
--- that are responsible for the dependence between $A_x$ and $A_y$ --- 
do not influence the joint probabilities $\pr{A_x \cap A_y}$ by much.
In order to be able to control the collision probabilities, we add
some restrictions to the event $A_x$. Instead of only looking for
long paths in the tree, we look for paths that maintain large
enough labels at each step. See equation \eqref{eq:eventa} for a
definition.
The probability of such an event can be bounded (Lemma
\ref{lem:permute}) using a rotation
argument introduced by \citet{And53} and \citet{Dwa69} and used in
the context of random trees by \citet{DR95}.

To simplify the presentation, the proof is carried out first for the case
where $X$ has a bounded density and possibly a mass at $0$, i.e.,
\begin{equation}\label{eq:densityatom}
X = \left\{ \begin{array}{ll}
\wt{X} & \textrm{with probability $1-p$}\\
0 & \textrm{with probability $p$},
\end{array} \right.
\end{equation}
where $\law(\wt{X})$ has a bounded density on $(0,1)$ and $p \in [0,1]$.
The reason we allow $X$ to have an atom at $0$ is to later handle
attachment distributions having unbounded densities (Theorem
\ref{thm:heightunbounded}).


\noindent \textbf{Preliminary lemmas.}
We begin by stating precise bounds on the probabilities of events of the form
$\event{X_1 \cdots X_t \geq b}$.

\begin{proposition}[\citet{Cra38}, see also \citet{DZ98}, chapter 2, page 27] \label{prop:cramer}\mbox{}
Let $Y_1, \dots, Y_t$ be a sequence of independent real random
variables distributed as $Y$ and having a well-defined expected
value $\ex{Y} \in \R \cup \{\pm \infty\}$. For any constant $a \in
\R$, we have
\begin{align*}
\pr{ Y_1 + \dots + Y_t \geq t a } &= \exp{-t \Lambda^*_Y(a) + o(t) \Big.}
    &\quad\text{if $a \geq \ex{Y}$ and $\ex{Y} \neq +\infty$}, \\
\pr{ Y_1 + \dots + Y_t \leq t a }  &= \exp{-t \Lambda^*_Y(a) + o(t) \Big.}
    &\quad\text{if $a \leq \ex{Y}$ and $\ex{Y} \neq -\infty$},
\end{align*}
where $\Lambda^*_Y$ is as defined in equation \eqref{eq:lambdaY}.
\end{proposition}

Before stating the corollary that we need, we define the rate
function $\Lambda^*$ for a random variable $\log X$ that has an
atom at $-\infty$. The function $\varphi :\lambda \mapsto \lambda
z - \log \ex{e^{\lambda \log X}}$ is well defined for $\lambda >
0$. We extend it for $\lambda = 0$ by $\varphi(0) = -
\log(1-\pr{\log X = -\infty})$. Then, $\Lambda^*$ is defined by
\begin{equation}\label{eq:lambdainfdef}
\Lambda^*(z) = \sup_{\lambda \geq 0} \{\varphi(\lambda)\}
\end{equation}
for all real $z \geq \ex{\log X}$. Note that this definition
coincides with the definition given in \eqref{eq:lambda} if $\pr{X
= 0} = 0$.
\begin{corollary}\label{cor:cramerup}
Let $X$ have an atom at $0$ with mass $p$ and any distribution on
$(0,1)$ with total mass $1-p$.  Let $X_1,\dots, X_t$ be i.i.d.\
random variables distributed as $X$. Then,
\[
\pr{ X_1 \cdots X_t \geq e^{ta} } = \exp{-t \leg{a} + o(t) \Big.}\quad
\left\{ \begin{array}{ll}
\text{for $a \geq \ex{\log X}$} & \text{if $\ex{\log X} > -\infty$}\\
\text{for $a \in \R$} & \text{if $\ex{\log X} = -\infty$}.
\end{array} \right.
\]
\end{corollary}
\begin{proof}
First if $p = 0$, we can apply Cram\'{e}r's theorem to $\log X$
and get the desired result. In what follows, assume $p > 0$ so
that $\log X = -\infty$ with positive probability. Let $t > 0$ be
integer, and let $\wt{X}_1, \ldots, \wt{X}_t$ be $t$ independent
random variables having the distribution of $X$ conditioned in $X
> 0$. If any $X_i=0$, $1\leq i \leq t$, then the product
$X_1\cdots X_t = 0$, and thus
\begin{align*}
\pr{ X_1 \cdots X_t \geq e^{ta} }
    &= (1-p)^t\; \pr{ \wt{X}_1 \cdots \wt{X}_t \geq e^{ta} } \\
    &= (1-p)^t\; \pr{ \log \wt{X}_1 + \dots + \log \wt{X}_t \geq ta }.
\end{align*}
For $a \geq \ex{\log \wt{X}}$, we get
\begin{align*}
\pr{ X_1 \cdots X_t \geq e^{ta} }
    &= (1-p)^t\; \exp{-t \Lambda^*_{\log \wt{X}}(a) + o(t) \Big.} \\
    &= \exp{-t \left(\Lambda^*_{\log \wt{X}}(a) - \log(1-p) \right) + o(t) \Big.}.
\end{align*}
Then, assume $\ex{\log \wt{X}} > -\infty$ and $a < \ex{\log
\wt{X}}$. Using the law of large numbers for $\log X$, we get
\[
\lim_{t \to \infty} \pr{ X_1 \cdots X_t \geq e^{ta} } = 1.
\]
Thus,
\[
(1-p)^t\left(1-o(1) \right) \leq \pr{ X_1 \cdots X_t \geq e^{ta} } \leq (1-p)^t
\]
which implies
\[
\pr{ X_1 \cdots X_t \geq e^{ta} } = \exp{t \log(1-p) + o(t)}.
\]
It only remains to show that
\begin{equation}\label{eq:lambdainf}
\Lambda^*(z) = \left\{ \begin{array}{ll}
\Lambda^*_{\log \wt{X}}(z) - \log(1-p) & \textrm{for $z \geq \ex{\log \wt{X} } $}\\
-\log(1-p) & \textrm{for $z \leq \ex{\log \wt{X}}$}.
\end{array} \right.
\end{equation}
Let $U$ be a random variable uniformly distributed on $(0,1)$ and
independent from $X$ and $\wt{X}$. Consider the event
$A=\event{U\leq p}$. Then $X\eqlaw 0\ind{A} + \wt{X}\ind{A^c}$.
Thus, for $z \geq \ex{\log X}$, we have
\begin{align*}
\sup_{\lambda > 0} \left\{\lambda z - \log \ex{X^{\lambda}} \right\}
&= \sup_{\lambda > 0} \left\{\lambda z - \log \ex{(0\ind{A} + \wt{X}\ind{A^c})^{\lambda}} \right\} \\
&= \sup_{\lambda > 0} \left\{\lambda z - \log \ex{\wt{X}^{\lambda}\ind{A^c} } \right\} \\
&= \sup_{\lambda > 0} \left\{\lambda z - \log\left(\ex{\wt{X}^{\lambda}} \ex{\ind{A^c}}\right) \right\} \\
&= \sup_{\lambda > 0} \left\{\lambda z - \log \left( \ex{\wt{X}^{\lambda}} (1-p) \right) \right\} \\
&= \sup_{\lambda > 0} \left\{\lambda z - \log \ex{\wt{X}^{\lambda}} \right\} - \log(1-p).
\end{align*}
As a result, using the definition \eqref{eq:lambdainfdef}, we obtain
\begin{align*}
\Lambda^*(z) &= \max \left\{ \sup_{\lambda > 0} \left\{\lambda z - \log \ex{\wt{X}^{\lambda}} \right\} - \log(1-p), -\log(1-p) \right\} \\
            &= \sup_{\lambda \geq 0} \left\{\lambda z - \log \ex{\wt{X}^{\lambda}} \right\} - \log(1-p).
\end{align*}
which matches the expression \eqref{eq:lambdainf} using Proposition \ref{prop:lambdastar}.
\end{proof}

The next lemma is based on a rotation argument introduced by \citet{And53} and \citet{Dwa69}.

\begin{lemma}\label{lem:permute}
Let $t$ be a positive integer, let $\beta > 0$, and let $X_1,
\dots, X_t$ be a sequence of non-negative independent and
identically distributed random variables. Then
\[
\pr{X_1 \geq \beta, X_1X_2 \geq \beta^{2}, \dots, X_1 \cdots X_t \geq \beta^{t} }
\geq \frac{1}{t} \pr{ X_1 \cdots X_t \geq \beta^{t} }.
\]
\end{lemma}
\begin{proof}
As $X_1, \dots, X_t$ are i.i.d., we can circularly continue the
indices: $Y_a = Y_{a+t} = \frac{X_a}{\beta}$ for all $a \in \{1,
\dots, t\}$. Then,
\begin{align*}
\pr{X_1 \geq \beta, \dots , X_1 \cdots X_t \geq \beta^t }
&= \pr{Y_1 \geq 1, \dots, Y_{1} \cdots Y_{t} \geq 1 } \\
&= \pr{Y_{a+1} \geq 1, \dots, Y_{a+1} \cdots Y_{a+t} \geq 1}
\end{align*}
for all $a \in \{1, \dots, t\}$ since the variables are i.i.d.\

Define $a \in \{1, \dots, t\}$ as the first minimum of $Y_1 \cdots
Y_a$. Then $Y_1 \cdots Y_t \geq 1$ implies that for all $b \in
\{1, \dots, t\}$,
\[
Y_{a+1} \cdots Y_{a+b} = \frac{Y_1 \cdots Y_{a+b}}{Y_1 \cdots Y_a} \geq 1.
\]
If $a+b \leq t$, the inequality holds by our choice of $a$. For
$a+b > t$, it can be seen by writing $Y_1 \cdots Y_{a+b} = Y_1
\cdots Y_t \cdot Y_{1} \cdots Y_{a+b-t}$ and using that $Y_1\cdots
Y_t \geq 1$.  Thus,
\[
\event{Y_1 \cdots Y_t \geq 1} \subseteq \bigcup_{a=1}^t \event{Y_{a+1} \geq 1, \dots, Y_{a+1} \cdots Y_{a+t} \geq 1}.
\]
So we have
\[
\pr{Y_1 \cdots Y_t \geq 1} \leq t \cdot \pr{Y_1 \geq 1, \dots, Y_1 \cdots Y_t \geq 1}. \qedhere
\]
\end{proof}


\noindent \textbf{Proof of the lower bound.}
For convenience of notation, the nodes of the tree are labeled
from $0$ to $3n$, and we shall study the height $H_{3n}$. For a
node $x \in \{2n+1, \dots, 3n\}$, $t \in \mathbb{N}$ and $0 <
\beta < 1$, define the event
\begin{equation}\label{eq:eventa}
A_{x,t}(\beta) = \event{ L(x,1) \geq n\beta, L(x,2) \geq n\beta^{2}, \dots, L(x,t) \geq n\beta^{t}}.
\end{equation}
We set $A_{x,0}(\beta)=\event{L(x,0) > n\beta^0}=\event{x > n}$ so
that $\pr{A_{x,0}(\beta)} = 1$. Note that when $\beta$ is clear
from the context, we just write $A_{x,t}$ for $A_{x,t}(\beta)$.

\begin{lemma}\label{lem:proba}
Assume $\law(X)$ is not a single mass. Let $c \in (1/\mu,
\alpha_{\max})$, $\beta = e^{-1/c}$ and $\delta > 0$ such that
$\Psi(c) + \delta < 1$ and $\Psi(c) - \delta > 0$.
Then there exists $t_0 = t_0(c, \delta, \law(X))$ such that  for
all integers $t \geq t_0$, $n \geq t \beta^{-t}$ and $2n+1 \leq x
\leq 3n$,
\[
\frac{\beta^{t}}{t} \leq \frac{\beta^{(\Psi(c)+\delta)t}}{t} \leq \pr{A_{x,t}(\beta)} \leq \beta^{(\Psi(c)-\delta) t}.
\]
\end{lemma}
\begin{proof}
First, using Proposition \ref{prop:psi} in Appendix
\ref{sec:applambda}, we know that $0<\Psi(c)<1$ for $c \in (1/\mu,
\alpha_{\max})$. So we can choose $\delta > 0$ with $\Psi(c) +
\delta < 1$ and $\Psi(c) - \delta > 0$.

We start with the upper bound. Using the same computation as in
the previous section,
\begin{align*}
\pr{L(x,t) \geq n\beta^{t}} &\leq \pr{ 3n X_{L(x,0)} \dots X_{L(x,t-1)} \geq n \beta^t} \\
    &= \pr{ 3 \beta^{-t} X_{L(x,0)} \dots X_{L(x,t-1)} \geq 1} \\
    &\leq \inf_{\lambda \geq 0} \exp{ \lambda (-t \log \beta + \log 3) + \Lambda(\lambda)t \Big.} \\
    &= \exp{- t \leg{-\frac{1}{c} - \frac{\log 3}{t}} }.
\end{align*}
By definition of $\Psi$, we have $\leg{-1/c} = \Psi(c)/c$. Thus
for $t$ large enough, by continuity of $\Lambda^*$, $\leg{-1/c -
(\log 3)/t} > (\Psi(c)-\delta)/{c}$. Thus,
\[
\pr{L(x,t) \geq n\beta^t } \leq \exp{-t(\Psi(c)-\delta)/c} = \beta^{(\Psi(c)-\delta)t}.
\]
To prove a lower bound on the probability of $A_{x,t}$, we use
that for all $s\in \{1,\dots,t\}$
\begin{align*}
\event{ L(x,s) \geq n\beta^s }
    &\supseteq \event{ 2nX_{L(x,0)}\cdots X_{L(x,s-1)} - s \geq n\beta^s } \\
    &\supseteq \event{ X_{L(x,0)}\cdots X_{L(x,s-1)} \geq \frac{\beta^s}{2} + \frac{s}{2n} } \\
    &\supseteq \event{ X_{L(x,0)}\cdots X_{L(x,s-1)} \geq \beta^s }.
\end{align*}
The last inclusion holds because we assumed $n \geq t \beta^{-t}
\geq s \beta^{-s}$ for all $s\leq t$. Thus, we write
\begin{align*}
\pr{A_{x,t}} &=  \pr{ L(x,1) \geq n\beta, L(x,2) \geq n\beta^2, \dots, L(x,t) \geq n\beta^t } \\
&\geq \pr{ X_{L(x,0)} \geq \beta, X_{L(x,0)}X_{L(x,1)} \geq \beta^2, \dots, X_{L(x,0)}\cdots X_{L(x,t-1)} \geq \beta^t }.
\end{align*}
We now use Lemma \ref{lem:permute} to get
\[
\pr{A_{x,t}} \geq \frac{1}{t} \pr{X_{L(x,0)} \cdots X_{L(x,t-1)} \geq \beta^t}.
\]
Using Corollary \ref{cor:cramerup} of Cram\'{e}r's theorem,
\begin{align*}
\pr{X_{L(x,0)} \cdots X_{L(x,t-1)} \geq \beta^t} &= \pr{X_{L(x,0)} \cdots X_{L(x,t-1)} \geq e^{-t/c}} \\
    &= \exp{-t \leg{-1/c} + o(t) \Big.}.
\end{align*}
But $\leg{-1/c} = \Psi(c)/c < (\Psi(c)+\delta)/c$. So for $t$ large enough,
\[
\pr{X_{L(x,0)} \cdots X_{L(x,t-1)} \geq \beta^{t}} \geq \exp{-(\Psi(c)+\delta)t/c} = \beta^{(\Psi(c)+\delta)t}.
\]
As a result
\[
\pr{A_{x,t}} \geq \frac{\beta^{(\Psi(c)+\delta)t}}{t} \geq \frac{\beta^t}{t}. \qedhere
\]
\end{proof}

Theorem \ref{thm:height} is proven using the second moment method 
on the number of nodes that have a large depth. 

\begin{lemma}\label{lem:jointproba}
Let $X$ have an atom of weight $p$ at $0$ for some $p\in [0,1)$,
and a density bounded by $\kappa$, of total mass $1-p$, on
$(0,1)$. Let $x \neq y$ be elements of $\{2n+1, \dots, 3n\}$, let
$t$ be a positive integer and let $\beta \in (0,1)$. Then
\[
\pr{A_{x,t} \cap A_{y,t}} \leq \sum_{s=0}^{t-1} \pr{A_{x,t}} \pr{A_{y,s}} \frac{(t+1) \kappa}{n \beta^{s}} + \pr{A_{x,t}} \pr{A_{y,t}}.
\]
\end{lemma}
\begin{proof}
If $v$ is a node of a \textsc{sarrt}, let $P_t(v) = \{L(v,0),
L(v,1), \dots, L(v,t)\}$ be the first $t+1$ elements of the
(random) path connecting $x$ to the root of the tree. Given $x$
and $y$, define $T = +\infty$ if $P_t(x)\cap P_t(y) = \emptyset$,
otherwise set $T$ to be the minimum non-negative $s$ such that
$L(y,s+1) \in P_t(x)$.
Then
\[
\pr{A_{x,t} \cap A_{y,t}} = \sum_{s=0}^{t-1} \pr{T = s, A_{x,t}\cap A_{y,t}} + \pr{T = +\infty, A_{x,t}\cap A_{y,t}}.
\]

In order to evaluate this expression, we fix the path $P_t(x)$
from $x$ to its $t$-th ancestor. Let $\F=\{Q\subseteq \{0, \dots,
3n\}: x = \max Q, \, |Q|\leq t\}$ be the set of possible paths.
For all $s \in \{0, \dots, t-1\}$
\begin{align*}
\pr{ T = s, A_{x,t}\cap A_{y,t} }
&= \sum_{Q\in\F} \pr{T=s, A_{x,t}\cap A_{y,t}, P_t(x) = Q} \\
&\leq \sum_{Q\in\F} \ind{A_{x,t}}(Q)\; \pr{T=s, A_{y,s}, P_t(x) = Q}
\end{align*}
where $\ind{A_{x,t}}(Q)$ is the indicator of the event $A_{x,t}$
when $P_t(x) = Q$. As the event $A_{x,t}$ is completely determined
by the path $P_t(x)$, $\ind{A_{x,t}}(Q)$ is deterministic.
\begin{align*}
&\pr{ T = s, A_{x,t}\cap A_{y,t} } \\
&\quad\leq \sum_{Q\in\F} \ind{A_{x,t}}(Q)\; \pr{P_{s}(y) \cap Q = \emptyset, L(y,s+1)\in Q, A_{y,s}, P_t(x) = Q} \\
&\quad= \sum_{Q\in\F} \ind{A_{x,t}}(Q)\; \sum_{u:\;\substack{u \geq n \beta^{s} \\ u \notin Q}}
    \pr{P_{s}(y) \cap Q = \emptyset, L(y,s) = u, \floor{uX_u} \in Q, A_{y,s}, P_t(x) = Q}.
\end{align*}
In order to simplify this expression, we use the independence
claim below.
\begin{claim}
For any $Q \subseteq \{0, \dots, 3n\}$ and $u \notin Q$, the
events $\event{P_{s}(y) \cap Q = \emptyset, L(y,s) = u, A_{y,s}}$,
$\event{\floor{uX_u} \in Q}$ and $\event{P_t(x) = Q}$ are mutually
independent.
\end{claim}
\begin{proof}
We show that the three events live in independent sigma-algebras.
Recall that an event $E$ is said to be in the sigma-algebra
generated by a random variable $Y$ when knowing the value of $Y$
determines whether $E$ holds or not.
\begin{enumerate}[(i)]
\item $\event{P_{s}(y) \cap Q = \emptyset, L(y,s) = u, A_{y,s}}$
is in the sigma-algebra generated by $\{X_w: w \notin Q, w \neq
u\}$. In fact, starting at $y$, it is possible to determine the
path of length $s$ starting at $y$ until it reaches a node in $Q
\cup \{u\}$. If any node in $Q$ is reached before $s$ steps, then
$\event{P_{s}(y) \cap Q = \emptyset}$ cannot hold. Moreover, if
node $u$ is reached before $s$, $\event{L(y,s)=u}$ cannot hold
because $u$ is not the root and the attachment distribution
$\law(X)$ is smaller than $1$. Otherwise, knowing the path
$P_{s}(y)$, it is easy to determine whether $\event{P_{s}(y)
\cap Q = \emptyset, L(y,s) = u, A_{y,s}}$ holds or not.

\item $\event{\floor{uX_u} \in Q}$ is in the sigma-algebra
generated by $X_u$.

\item $\event{P_t(x) = Q}$ is in the sigma-algebra generated by
$\{X_w: w \in Q\}$, using an argument similar to (i).
\end{enumerate}
We conclude by recalling that the random variables $X_0, X_1,
\dots, X_{3n}$ are independent.
\end{proof}
It follows that
\begin{align*}
&\pr{ T = s, A_{x,t}\cap A_{y,t} } \\
&\leq \sum_{Q\in\F} \ind{A_{x,t}}(Q) \sum_{u:\; \substack{u \geq n \beta^{s} \\ u \notin Q}}
    \pr{P_{s}(y) \cap Q = \emptyset, L(y,s) = u, A_{y,s}} \pr{P_t(x) = Q} \pr{\floor{uX_u} \in Q} \\
&\leq \sum_{Q\in\F} \ind{A_{x,t}}(Q) \pr{A_{y,s}} \pr{P_t(x) = Q}
    \sup_{u:\; \substack{u \geq n \beta^{s} \\ u \notin Q}} \pr{\floor{uX_u} \in Q} \\
&\leq \left(\sum_{Q\in\F} \ind{A_{x,t}}(Q) \pr{P_t(x) = Q} \right) \pr{A_{y,s}} (t+1)
    \sup_{\substack{u:\; u \geq n \beta^{s} \\ w:\; w \geq n \beta^t}} \pr{\floor{uX_u} = w} \\
&= \pr{A_{x,t}} \pr{A_{y,s}} (t+1)
    \sup_{\substack{u:\; u \geq n \beta^{s} \\ w:\; w \geq n \beta^t}} \pr{\floor{uX_u} = w}.
\end{align*}
The last inequality holds because when the event $A_{x,t}$ holds,
all nodes in $P_t(x)$ have a label at least $n \beta^t$.
In order to bound the collision probability $\pr{\floor{uX_u } =
w}$, we first notice that $w > 0$. So we can use the fact that
conditioned on $X > 0$, $X$ has a density bounded by $\kappa$:
\begin{align*}
\pr{\floor{uX_u } = w} \leq  \pr{ X_u \in\left[\frac{w}{u},\frac{w+1}{u}\right) } \leq \frac{\kappa}{u}.
\end{align*}
Thus,
\begin{align*}
\pr{T = s, A_{x,t}\cap A_{y,t}} &\leq \pr{A_{x,t}} \pr{A_{y,s}} \frac{(t+1)\kappa}{n \beta^{s}}.
\end{align*}
Repeating the above argument for $T=+\infty$, we get
\begin{align*}
\pr{T = +\infty, A_{x,t}\cap A_{y,t}}
    &\leq \sum_{Q\in\F} \ind{A_{x,t}}(Q)\; \pr{P_{t}(y) \cap Q = \emptyset, A_{y,t}, P_t(x) = Q} \\
    &\leq \left(\sum_{Q\in\F} \ind{A_{x,t}}(Q) \pr{P_t(x) = Q} \right) \pr{A_{y,t}} \\
    &= \pr{A_{x,t}} \pr{A_{y,t}}. \qedhere
\end{align*}
\end{proof}

\begin{theorem}\label{thm:height}
Let there exist $p\in[0,1]$ such that with probability $p$,
$X$ has an atom at $0$, and with
probability $1-p$, $X$ has a bounded density on $[0,1)$. The height $H_n$ of a
\textsc{sarrt} with attachment $X$ satisfies
\[
\frac{H_n}{\log n} \inprob \alpha_{\max} \quad\text{as}\quad n\to\infty,
\]
where $\alpha_{\max}$ is defined in equation \eqref{eq:alphamax}.
\end{theorem}
\begin{proof}
If the atom at $0$ has probability $1$, then $H_n = 1$ and
$\alpha_{\max} = 0$. In the rest of the proof, we assume that $X$
is not a single mass. Fix $\delta \in (0,1/2)$, $\eps \in (0,1)$
with $3\delta < \eps$ and $c \in (1/\mu, \alpha_{\max})$. Define
$\beta = e^{-1/c}$ and $t = \floor{(1-\eps) c \log n}$. Our
objective is to show that
\[
\lim_{n \to \infty} \pr{H_{3n} \geq t} = 1.
\]
For this we consider the event
\[
\event{\bigcup_{x=2n+1}^{3n} A_{x,t}}
\]
where the events $A_{x,t}$ are defined in equation
\eqref{eq:eventa}. The fact that $A_{x,t}$ holds implies that
$L(x,t) \geq n\beta^t \geq n/n^{1-\eps} = n^{\eps} \geq 1$, i.e.,
the depth of node $x$ is at least $t$. A lower bound on the
probability is given by the following second moment inequality
\citep{CE52}:
\begin{equation}\label{eq:secondmoment}
\pr{\bigcup_{x=2n+1}^{3n} A_{x,t}} \geq \frac{\left(\sum_{x=2n+1}^{3n} \pr{A_{x,t}}\right)^2 }{\sum_{x=2n+1}^{3n} \pr{A_{x,t}} + \sum_{x \neq y} \pr{A_{x,t} \cap A_{y,t}}}.
\end{equation}
The symbol $\sum_{x \neq y}$ is used instead of
$\sum_{x=2n+1}^{3n} \sum_{y=2n+1, y \neq x}^{3n}$ to keep the
notation light. Let $t_0(c,\delta, \law(X))$ be defined as in
Lemma \ref{lem:proba}. When $n$ is large enough, the conditions $t
\geq t_0$ and $n \geq t \beta^{-t}$ are met. So Lemma
\ref{lem:proba} gives
\begin{equation}\label{eq:probalow}
\pr{A_{x,t}} \geq \frac{\beta^t}{t} \geq \frac{1}{tn^{1-\eps}}.
\end{equation}
Now, fixing $x \neq y$, we have by Lemma \ref{lem:jointproba}:
\[
\pr{A_{x,t} \cap A_{y,t}} \leq \sum_{s=0}^{t-1} \pr{A_{x,t}} \pr{A_{y,s}} \frac{(t+1) \kappa}{n \beta^{s}} + \pr{A_{x,t}} \pr{A_{y,t}}.
\]
For $s \geq t_0$, we apply Lemma \ref{lem:proba} to find an upper bound on $\pr{A_{x,s}}$:
\begin{align}\label{eq:jointprobacalc}
\pr{A_{x,t}\cap A_{y,t}}
    &\leq \pr{A_{x,t}} \left( \sum_{s=0}^{t_0-1} \frac{(t+1)\kappa}{n \beta^{s}} +
        \sum_{s=t_0}^{t-1} \beta^{(\Psi(c)-\delta)s} \frac{(t+1)\kappa}{n \beta^{s}} + \pr{A_{y,t}} \right) \notag \\
    &\leq \pr{A_{x,t}} \left( \frac{(t+1)\kappa}{n} \cdot \frac{\beta^{-t_0}-1}{\beta^{-1} - 1} +
        \frac{(t+1)\kappa}{n} \sum_{s=0}^{t-1} \beta^{(\Psi(c)-\delta - 1) s} + \pr{A_{y,t}} \right) \notag \\
    &\leq \pr{A_{x,t}}\left( \Ord{\frac{t}{n}} + \frac{(t+1)\kappa}{n} \cdot \frac{\beta^{(\Psi(c)-\delta - 1)t}-1}{\beta^{(\Psi(c)-\delta-1)}-1}
        + \pr{A_{y,t}}\right).
\end{align}
We now show that the dominating term is $\pr{A_{x,t}}
\pr{A_{y,t}}$. Using inequality \eqref{eq:probalow},
\begin{equation}\label{eq:firstterm}
\frac{t/n}{\pr{A_{y,t}}} \leq \frac{t^2 n^{1-\eps} }{n} = \Ord{n^{-\eps/2}}
\end{equation}
as $t = O(\log n)$. Moreover, using the more precise lower bound
on $\pr{A_{y,t}}$ given in Lemma \ref{lem:proba},
\[
\frac{t \beta^{(\Psi(c)-\delta-1)t} }{n \pr{A_{y,t}} }
\leq \frac{t^2 \beta^{(\Psi(c)-\delta-1)t} \beta^{-(\Psi(c)+\delta)t} }{n}
= \frac{t^2 (\beta^{-t})^{2\delta} \beta^{-t}}{n}.
\]
By definition of $t$, $\beta^{-t} \leq n^{1-\eps}$, and thus
\begin{equation}\label{eq:secondterm}
\frac{t \beta^{(\Psi(c)-\delta-1)t} }{n \pr{A_{y,t}} } \leq t^2 n^{2\delta - \eps}
\leq t^2 n^{-\eps/3} = O(n^{-\eps/4}).
\end{equation}
Plugging inequalities \eqref{eq:firstterm} and
\eqref{eq:secondterm} into \eqref{eq:jointprobacalc}, we get
\[
\pr{A_{x,t} \cap A_{y,t}} \leq \pr{A_{x,t}} \pr{A_{y,t}} \left(1 + \Ord{n^{-\eps/4}}\right).
\]
Taking the sum over all nodes $x \neq y$ with $x,y \in \{2n+1,
\dots, 3n\}$, we obtain
\[
\sum_{x \neq y} \pr{A_{x,t}\cap A_{y,t}} \leq \left(\sum_{x=2n+1}^{3n} \pr{A_{x,t}}\right)^2\left(1 + \Ord{n^{-\eps/4}} \right).
\]
Moreover, using inequality \eqref{eq:probalow}, we have
\[
\sum_{x=2n+1}^{3n} \pr{A_{x,t}} \geq n\frac{1}{tn^{1-\eps}} = \frac{n^{\eps}}{t}.
\]
Thus, plugging these bounds in \eqref{eq:secondmoment}, we get
\begin{align*}
\pr{\bigcup_{x=2n+1}^{3n} A_{x,t}}
    &\geq \frac{1}{\left(\sum_{x=2n+1}^{3n} \pr{A_{x,t}} \right)^{-1} + 1 + \Ord{n^{-\eps/4}}} \\
    &\geq 1 - \Ord{n^{-\eps/4}} - \Ord{tn^{-\eps}}.
\end{align*}

This shows that
\begin{equation}\label{eq:probheightlow}
\pr{H_{3n} \geq t} = \pr{H_{3n} \geq \floor{(1-\eps)  c \log n}} \geq 1 - \Ord{n^{-\eps/4}}.
\end{equation}
We conclude that for any $\eps > 0$,
\[
\lim_{n\to\infty} \pr{H_{n} \geq (1-\eps)  \alpha_{\max} \log n } = 1.
\]
Combining this with the upper bound proved in Lemma
\ref{lem:upper}, we get the desired result.
\end{proof}


\subsection{Attachment distribution with unbounded density}\label{sec:unbounded}

In order to handle attachment distributions $X$ having unbounded densities,
the next lemma shows that we can approximate $X$ by $X_{\delta}$
that has bounded density and an atom at $0$.

\begin{lemma}\label{lem:unbounded}
Assume that $X \in [0,1)$ has a density, and let $z \geq -\mu$ be
such that $\leg{z} < +\infty$. Then for all $\delta > 0$, there
exists $X_{\delta} \leq X$ such that $\law(X_{\delta})$ has a
bounded density and an atom at $0$, such that
\[
\Lambda^*(z) \leq \Lambda_{\delta}^*(z) \leq \Lambda^*(z) + \delta
\]
where $\Lambda_{\delta}^{*}$ is defined as in
\eqref{eq:lambdainfdef} for $X_{\delta}$.
\end{lemma}
\begin{proof}
The constants $\eta,b > 0$ will be chosen later. Let $f$ be the
density of $\law(X)$ and define the event $A=\event{f(X) > b}$.
Take $b$ be such that $\pr{A} \leq \eta$. Define $X_{\delta} =
0\ind{A} + X\ind{A^c}$. We have
\begin{align*}
\Lambda_{\delta}^*(z)
    &= \sup_{\lambda \geq 0} \left\{ \lambda z - \log \ex{X_{\delta}^{\lambda}} \right\} \\
    &= -\log \inf_{\lambda \geq 0} \left\{ e^{-\lambda z} \ex{(0\ind{A} + X\ind{A^c})^{\lambda}} \right\} \\
    &= -\log \inf_{\lambda \geq 0} \left\{ e^{-\lambda z} \ex{0\ind{A} + X^{\lambda}\ind{A^c}} \right\} \\
    &= -\log \inf_{\lambda \geq 0} \left\{ e^{-\lambda z} \left(\ex{X^{\lambda}} - \ex{X^{\lambda}\ind{A}}\right) \right\}.
\end{align*}
Note that the expression $\lambda z - \log
\ex{X_{\delta}^{\lambda}}$ is understood to evaluate to
$-\log(1-\pr{A})$ for $\lambda = 0$ as in equation
\eqref{eq:lambdainfdef}. Trivially, we first get
$\Lambda_{\delta}^*(z) \geq \Lambda^*(z)$. Moreover, using
Cauchy-Schwarz inequality,
\[
\ex{X^{\lambda}\ind{A}} \leq \sqrt{\ex{X^{2\lambda}}} \sqrt{\pr{A}}
\leq \sqrt{\ex{X^{2\lambda}}} \sqrt{\eta}.
\]
Thus,
\begin{align*}
\Lambda_{\delta}^*(z)
    &\leq -\log \inf_{\lambda \geq 0} \left\{ e^{-\lambda z} \ex{X^{\lambda}} - \sqrt{e^{-2\lambda z} \ex{X^{2\lambda}}} \sqrt{\eta} \right\} \\
    &\leq -\log \left( \inf_{\lambda \geq  0} \left\{ e^{-\lambda z} \ex{X^{\lambda}} \right\} - \sqrt{\eta}\inf_{\lambda \geq 0} \left\{ \sqrt{ e^{-2\lambda z} \ex{X^{2\lambda}}} \right\} \right) \\
    &= -\log \left( e^{-\leg{z}} - \sqrt{\eta} e^{-\leg{z}/2} \right) \\
    &= \leg{z} - \log\left(1-\sqrt{\eta e^{\leg{z}}}\right).
\end{align*}
By choosing $\eta$ so that $\log\left(1-\sqrt{\eta
e^{\leg{z}}}\right) \leq \delta$, we obtain the desired result.
\end{proof}

We can now restate the theorem for any density.

\begin{theorem}\label{thm:heightunbounded}
Let there exist $p\in[0,1]$ such that with probability $p$,
$X$ has an atom at $0$, and with
probability $1-p$, $X$ has a bounded density on $[0,1)$. The height $H_n$ of a
\textsc{sarrt} with attachment $X$ satisfies
\[
\frac{H_n}{\log n} \inprob \alpha_{\max} \quad\text{as}\quad n\to\infty,
\]
where $\alpha_{\max}$ is defined in equation \eqref{eq:alphamax}.
\end{theorem}
\begin{proof}
If the atom has probability $1$, then Theorem \ref{thm:height} can
be applied. In the rest of the proof, we assume that the atom at
$0$ has weight less than one. Since Lemma \ref{lem:upper} does not
have any restrictions on the distribution $\law(X)$, we have for
any $\eps > 0$,
\[
\lim_{n \to \infty} \pr{H_n \geq (\alpha_{\max}+\eps) \log n} = 0.
\]

For the lower bound, we use Theorem \ref{thm:height} via the
transformation defined in Lemma \ref{lem:unbounded}. Let $\eps >
0$ and pick $\delta > 0$ small enough so that
$\Psi(\alpha_{\max}-\eps) + \alpha_{\max} \delta < 1$. This is
possible because $\Psi(\alpha_{\max}-\eps) < 1$ (Proposition
\ref{prop:psi} in Appendix \ref{sec:applambda}). Then define
$X_{\delta}$ as in Lemma \ref{lem:unbounded}, so that
$\Lambda^*(z) \leq \Lambda_{\delta}^*(z) \leq \Lambda^*(z) +
\delta$. Define a tree $\wt{T}_n$ with a sequence $\wt{X}_0,
\dots, \wt{X}_n$ of independent random variables distributed as
$X_{\delta}$. Using Theorem \ref{thm:height} for the tree
$\wt{T}_n$, we get in particular a lower bound on its height
$\wt{H}_n$:
\[
\lim_{n \to \infty} \pr{\wt{H}_n \leq (\wt{\alpha}_{\max}-\eps) \log n} = 0
\]
where $\wt{\alpha}_{\max} = \inf\left\{c:\; c > \frac{1}{\mu}
\;\text{ and }\; \Psi_{\delta}(c) > 1\right\}$ and
$\Psi_{\delta}(c) = c \Lambda^*_{\delta}(-1/c)$. Recall that
$X_{\delta}$ as obtained from Lemma \ref{lem:unbounded} satisfies
$X_{\delta} \leq X$, which implies that $\wt{H}_n$ is
stochastically not larger than $H_n$. Thus,
\[
\lim_{n \to \infty} \pr{H_n \leq (\wt{\alpha}_{\max}-\eps) \log n} = 0.
\]
Next, if $\Psi$ is the function defined in \eqref{eq:psi} for the
(original) random variable $X$ and $\alpha_{\max} =
\inf\left\{c:\; c > \frac{1}{\mu} \;\text{ and }\; \Psi(c) >
1\right\}$, we have by construction of $X_{\delta}$,
\[
\Psi(\alpha_{\max}-\eps) \leq \Psi_{\delta}(\alpha_{\max}-\eps) \leq \Psi(\alpha_{\max}-\eps) +\alpha_{\max} \delta < 1.
\]
As a result, by definition of $\wt{\alpha}_{\max}$, we have
\[
\wt{\alpha}_{\max} \geq \alpha_{\max}-\eps
\]
so that
\[
\lim_{n \to \infty} \pr{H_n \leq (\alpha_{\max}-2\eps) \log n} = 0. \qedhere
\]
\end{proof}


\subsection{Almost sure convergence and convergence in mean}\label{sec:convergence}

Using Proposition \ref{prop:almostsure} below and the explicit
probability bounds given in the proofs of Lemma \ref{lem:upper},
equation \eqref{eq:probheightup} and Theorem \ref{thm:height},
equation \eqref{eq:probheightlow}, we get $\lim_{n\to\infty} 
\frac{H_n}{\log n} = \alpha_{\max}$ almost surely as stated above in Theorem \ref{thm:heightmain}.
We should mention that \citet{Pit94} also proved almost sure convergence of the height for the
\textsc{urrt}.

\begin{proposition}\label{prop:almostsure}
Let $H_n$ be a non-decreasing sequence of random variables and let
$\alpha \geq 0$ be such that for all $\eps >  0$,
\[
\pr{H_n \geq (\alpha +\eps) \log n} = \Ord{\frac{1}{\log n}} \quad
\text{ and } \quad \pr{H_n \leq (\alpha-\eps) \log n} =
\Ord{\frac{1}{\log n}}.
\]
Then, with probability $1$,
\[
\lim_{n \to \infty} \frac{H_n}{\log n} = \alpha.
\]
\end{proposition}
\begin{proof}
Let $\gamma \geq 3$ be an integer. We consider the maxima of the
sequence $H_n$ for $n$ in intervals of the form $[\gamma^{k^2},
\gamma^{(k+1)^2}]$ for positive integers $k$. For $\eps > 0$, we
have
\begin{align*}
\pr{\max_{\gamma^{k^2} \leq n \leq \gamma^{(k+1)^2}} \frac{H_n}{\log
    n} \geq (\alpha +\eps)} &\leq \pr{H_{\gamma^{(k+1)^2}} \geq (\alpha+\eps)
    \log \gamma^{k^2} } \\
&\leq \pr{ H_{\gamma^{(k+1)^2} } \geq (\alpha +\eps)
    \Big((k+1)^2 \log \gamma - (2k + 1) \log \gamma \Big) } \\
&= \pr{H_{\gamma^{(k+1)^2} } \geq (\alpha
+\eps) \log \left( \gamma^{(k+1)^2} \right) \left(1 - \frac{2k + 1}{(k+1)^2} \right) } \\
    &= \Ord{\frac{1}{\log \gamma^{(k+1)^2} } } = \Ord{\frac{1}{k^2}}.
\end{align*}
Using the Borel-Cantelli lemma, there exists $k_0$ such that,  $\max_{n
\geq \gamma^{k_0}} \frac{X_n}{\log n} \leq \alpha +\eps$ with
probability $1$. Similarly,
\[
\pr{\min_{\gamma^{k^2} \leq n \leq \gamma^{(k+1)^2}} \frac{H_n}{\log
n} \leq (\alpha-\eps)} = \Ord{\frac{1}{k^2}}.
\]
Thus, there exists $n_0$ such that for $n \geq n_0$, $\alpha-\eps \leq
\frac{H_n}{\log n} \leq \alpha +\eps$ almost surely.
\end{proof}

The next theorem shows that Theorem \ref{thm:heightunbounded}
implies the convergence of the sequence $\frac{\ex{H_n}}{\log n}$.
\begin{theorem}\label{thm:expheight}
Let there exist $p\in[0,1]$ such that with probability $p$,
$X$ has an atom at $0$, and with
probability $1-p$, $X$ has a bounded density on $[0,1)$. The height $H_n$ of a
\textsc{sarrt} with attachment $X$ satisfies
\[
\lim_{n \to \infty} \frac{\ex{H_n}}{\log n} = \alpha_{\max},
\]
where $\alpha_{\max}$ is defined in equation \eqref{eq:alphamax}.
\end{theorem}
\begin{proof}
For any $\eps > 0$,
\[
\ex{H_n} \geq (\alpha_{\max} - \eps) \log n \cdot \pr{H_n \geq (\alpha_{\max}-\eps) \log n}.
\]
Taking the limit as $n \to \infty$ and observing that the
inequality holds for any $\eps > 0$,
\[
\liminf_{n \to \infty} \frac{\ex{H_n}}{\log n} \geq \alpha_{\max}.
\]
For the upper bound, fix $\eps > 0$. We have
\begin{align*}
\ex{H_n} &\leq (\alpha_{\max} + \eps) \log n + 2 + \sum_{t = \ceil{(\alpha_{\max}+\eps) \log n + 2}}^{\infty} \pr{H_n \geq t} \\
        &\leq (\alpha_{\max} + \eps) \log n + \log n \cdot \sum_{i= 0}^{\infty} \pr{H_n \geq (\alpha_{\max}+\eps+i) \log n + 2}.
\end{align*}
The bound in equation \eqref{eq:probheightup} of Lemma
\ref{lem:upper} gives
\[
\pr{H_n \geq (\alpha_{\max}+\eps+i) \log n + 2} \leq n^{1-\Psi(\alpha_{\max}+\eps+i)}.
\]
But using the monotonicity of $\Lambda^*$ (Proposition
\ref{prop:lambdastar}),
\begin{align*}
\Psi(\alpha_{\max}+\eps+i) &= (\alpha_{\max}+\eps+i) \leg{-\frac{1}{\alpha_{\max}+\eps+i}} \\
                & \geq (\alpha_{\max}+\eps+i) \leg{-\frac{1}{\alpha_{\max}+\eps}} \\
                & \geq \frac{\alpha_{\max} + \eps+ i}{\alpha_{\max}}.
\end{align*}
In the last inequality, we used the definition of $\alpha_{\max}$
(equation \eqref{eq:alphamax}). Thus,
\begin{align*}
\ex{H_n} &\leq (\alpha_{\max} + \eps) \log n + 2+ (\log n) \cdot n^{1-\Psi(\alpha_{\max}+\eps)} + \log n \cdot \sum_{i= 1}^{\infty} n^{-i/\alpha_{\max}}.
\end{align*}
Finally,
\[
\limsup_{n \to \infty} \frac{\ex{H_n}}{\log n} \leq \alpha_{\max}+\eps. \qedhere
\]
\end{proof}


\section{The minimum depth}\label{sec:mindepth}

\red{
In the previous section, we considered the maximum depth or height of a tree. In this section, we study the minimum depth. Observe that considering the minimum depth over all the nodes is not interesting: $\min_{0\leq i \leq n} D_i = D_0 = 0$. Instead, we define the minimum depth by $M_n = \min_{n/2\leq i \leq n} D_i$. The reader will be easily convinced that the results remain unchanged if we consider $\min_{\delta n\leq i \leq n} D_i$ for some $\delta\in (0,1)$.

The objective of this section is to show that 
$\frac{M_n}{\log n} \to \alpha_{\min}$ almost surely where 
\begin{equation}\label{eq:alphamin}
\alpha_{\min} = 
\left\{ \begin{array}{ll}
0 & \textrm{if }  [0,1/\mu) \cap \{c: \; \Psi(c) > 1\} = \emptyset\\
\sup \left\{c:\; 0 \leq c < \frac{1}{\mu} \;\text{ and }\; \Psi(c) > 1\right\}  & \textrm{otherwise}
\end{array} \right.
\end{equation}
and $\Psi$ is defined as in equation \eqref{eq:psi} in Section \ref{sec:height}. 
Note that if $\mu = \ex{-\log X} = +\infty$, then $\alpha_{\min} = 0$, and 
$\frac{M_n}{\log n} \inprob 0$ using Theorem \ref{thm:depth}. In the sequel, 
we assume $\mu < +\infty$. In this case, provided that $X$ is not constant, 
Proposition \ref{prop:psi} in Appendix \ref{sec:applambda} implies that 
$\alpha_{\min} < 1/\mu$. The following theorem sums up the results we prove in this section.

\begin{theorem}\label{thm:mindepthmain}
The minimum depth $M_n$ of a \textsc{sarrt} with attachment $X$ having a density 
satisfies
\[
\frac{M_n}{\log n} \inprob \alpha_{\min},
\]
where $\alpha_{\min}$ is defined in equation \eqref{eq:alphamin}.
\end{theorem}
\begin{remark}
If $X = \alpha \in [0,1)$ with probability $1$, then $\alpha_{\min} = 1/\mu = -1/\log \alpha$ and it is easy to see that the results of the theorem also hold in this case.
\end{remark}

The proof of Theorem \ref{thm:mindepthmain} follows the same general 
idea as for the height with some complications for the upper bound. 
A lower bound on $M_n$ similar to the upper bound for the height 
(Section \ref{sec:heightup}) is given in next section. The proof of the 
upper bound is more delicate and it is the topic of Section \ref{sec:mindepthup}. Observe that $\frac{M_n}{\log n}$ does not converge almost surely as there are nodes with arbitrarily large labels that choose the root as a parent. 
}


\subsection{The minimum depth: lower bound}\label{sec:mindepthlow}

\begin{lemma}\label{lem:lower}
For any $c<\alpha_{\min}$, we have $\pr{M_n \leq c \log n} \to 0.$
\end{lemma}
\begin{proof}
If $\alpha_{\min} = 0$, then the lemma clearly holds. For $\alpha_{\min} > 0$, a calculation similar to that of Lemma \ref{lem:upper} shows that
\[
\pr{D_n \leq \floor{c\log n}} \leq \left(\frac{n}{1+\floor{c\log n}}\right)^{-\Psi(c)}
\]
using the definition of $\Psi$ (equation \eqref{eq:psi}).
By applying a union bound, we get a lower bound on the shortest path:
\begin{align*}
\pr{M_n \leq \floor{c\log n}} &= \pr{\min_{n/2\leq i \leq n} D_i \leq  \floor{c\log n}} \\
    &\leq n\, \pr{D_{\floor{n/2}} <  \floor{c\log n}} \\
    &= \Ord{n \cdot \left(\frac{n}{\log n}\right)^{-\Psi(c)}} \to 0.
\end{align*}
because $\Psi(c) > 1$ for $c < \alpha_{\min}$.
\end{proof}


\subsection{The minimum depth: upper bound}\label{sec:mindepthup}

In this section, we introduce the possibility for $X$ to have an atom at $+\infty$. 
This is needed only to take care of attachment distributions that have 
unbounded densities. A node $x$ for which $X_x = +\infty$ is attached to an 
imaginary node at $+\infty$, that does not have any ancestor, so that $L(x, s) =
+\infty$ for all $s \geq 1$. Even though such a choice of $X$ does
not fit in our definition of a \textsc{sarrt}, it is only used as
an auxiliary construction, and it is still possible to define all
the quantities that are based on $X$. We define $\Lambda^*$ 
for a random variable $\log X$ that has an atom at $+\infty$ as in 
the case of an atom at $-\infty$ (see equation \eqref{eq:lambdainfdef}):
\[
\Lambda^*(z) = \max \left\{ \sup_{\lambda < 0} \left\{ \lambda z - \log \ex{e^{\lambda \log X}} \right\}, -\log \left(1-\pr{X = +\infty}\right) \right\}
\]
for all $z \leq \ex{\log X}$. The function $\Psi$ is defined as in
equation \eqref{eq:psi}. We can then prove a statement analogous
to Corollary \ref{cor:cramerup} which we state below. 

\begin{corollary}\label{cor:cramerlow}
Let $X$ have an atom at $+\infty$ with mass $p\in [0,1)$ and any
distribution on $(0,1)$ with total mass $1-p$ such that $\ex{\log X}$ is well-defined. Let $X_1,\dots,
X_t$ be i.i.d.\ random variables distributed as $X$. Then,
\[
\pr{ X_1 \cdots X_t \leq e^{ta} }  = \exp{-t \leg{a} + o(t) \Big.}
\left\{ \begin{array}{ll}
\textrm{for $a \leq \ex{\log X}$} & \textrm{if $\ex{\log X} < +\infty$}\\
\textrm{for $a \in \R$} & \textrm{if $\ex{\log X} = +\infty$}.
\end{array} \right.
\]
\end{corollary}

\red{
Recall that for the height, we defined the event $A_{x,t}$ (equation \eqref{eq:eventa}) 
which captures the idea that the path up to the root originating from $x$ keeps 
large enough labels. By analogy, the corresponding event $B_{x,t}$ for the 
minimum depth is to have a path whose labels stay small in all steps. Given a 
design parameter $\beta \in (0,1)$,
\begin{equation}\label{eq:eventb}
B_{x,t}(\beta) = \event{ L(x,1) \leq 2n\beta, L(x,2) \leq 2n\beta^{2}, \dots, L(x,t) \leq 2n\beta^{t}}.
\end{equation}

The following lemma gives a bound on the probability of the event
$B_{x,t}$ assuming that $X$ has a bounded density and an atom 
at $+\infty$. The proof is based on a rotation argument and is
similar to that of Lemma \ref{lem:proba} with some minor modifications. 
Hence, we omit it to shorten the presentation.}

\begin{lemma}\label{lem:probb}
Let $X$ have an atom of weight $p\in [0,1)$ at $+\infty$, and any
distribution, of total mass $1-p$, on $(0,1)$. Moreover, assume
$\mu = \ex{-\log X}$ is well-defined and not $+ \infty$. Define $\theta = +\infty$ if
$\ex{-\log X} = -\infty$ (equivalently, if $p > 0$) and $\theta =
1/\mu$ otherwise. Let $c \in (\alpha_{\min}, \theta)$, $\beta =
e^{-1/c}$ and $\delta > 0$ such that $\Psi(c)+\delta < 1$. Then
there exists $t_0 = t_0(c, \delta, \law(X))$ such that  for all
integers $t \geq t_0$, $n \geq t\beta^{-t}$ and $n+1 \leq x \leq
2n$,
\[
\frac{\beta ^{t}}{t} \leq \frac{\beta^{(\Psi(c)+\delta)t}}{t} \leq \pr{B_{x,t}(\beta)} \leq \beta^{(\Psi(c)-\delta) t}.
\]
\end{lemma}

\red{
Next, we prove that there is enough independence between 
the events $B_{x,t}$ to allow us to use the second moment method.
In the context of the study of the height (Section \ref{sec:heightlow}), this is done for the events $A_{x,t}$ in Lemma \ref{lem:jointproba} where the probability 
of the event $\event{A_{x,t} \cap A_{y,t}}$ is bounded  
by estimating the probability of collisions. To obtain such a bound for the event $\event{B_{x,t} \cap B_{y,t}}$, the main difference is 
that we condition on the different intervals of labels where the collision might take place instead of the collision time $T$. This is because, unlike the event $A_{x,t}$ which gives a lower bound on the labels of the nodes in the path from node $x$ to the root, 
the event $B_{x,t}$ only implies an upper bound on the labels. Being able to bound from below the node labels is important to bound the collision probability.}

\begin{lemma}\label{lem:jointprobb}
Let $X$ have an atom of weight $p\in [0,1)$ at $+\infty$, and a
density bounded by $\kappa$, of total mass $1-p$, on $(0,1)$. Let
$x \neq y$ be elements of $\{n+1, \dots, 2n\}$, let $t$ be a
positive integer and let $\beta \in (0,1)$. Then
\[
\pr{B_{x,t} \cap B_{y,t}} \leq \sum_{s=1}^t \pr{B_{x,t}} \pr{B_{y,s-1}} \frac{(t+1) \kappa}{n \beta^{s-1}} + \pr{B_{x,t}} \pr{B_{y,t}}.
\]
\end{lemma}
\begin{proof}
We consider the collision time $T$ when the path starting at $y$
meets the path of $x$.  Define $T=+\infty$ if $P_t(x)\cap P_t(y) = \emptyset$ 
and $T=\min\{s \geq 0: L(y,s+1) \in P_t(x)\}$ otherwise. We introduce 
the random variables $T(x,i) = \min \{s \geq 0: L(x,s) \leq 2n\beta^i\}$. 
We have $\event{T(x,s)\leq s} = \event{L(x,s)\leq 2n\beta^s}$ for every $s$. 
In order to be able to bound collisions, instead of conditioning on a fixed value 
of $T$ we condition on $T$ being in some interval $I_s = \big[T(x,s-1),T(x,s)\big)$
or $I_{\infty} = \big[T(x,t),+\infty\big)$. If $T\in I_s$, then we know that the
collision happened between $n\beta^{s}$ and $n\beta^{s-1}$.
\[
\pr{B_{x,t} \cap B_{y,t}} = \sum_{s=1}^{t} \pr{T \in I_s, B_{x,t}\cap B_{y,t}} +
\pr{T \in I_{\infty}, B_{x,t}\cap B_{y,t}}.
\]
In order to evaluate this expression, we fix the path $P_t(x)$
from $x$ to its $t$-th ancestor and average over all possible
paths in $\F=\{Q\subseteq \{0, \dots, 3n\}: x = \max Q, \,
|Q|\leq t\}$. We have
\begin{align*}
&\pr{ T \in I_s, B_{x,t}\cap B_{y,t} } \\
&= \sum_{Q\in\F} \sum_{\ell=0}^{t-1} \pr{T=\ell, \ell \in I_s, B_{x,t}\cap B_{y,t}, P_t(x) = Q} \\
&\leq \sum_{Q\in\F} \sum_{\ell=0}^{t-1} \ind{B_{x,t}}(Q)\; \pr{ P_{\ell}(y) \cap Q = \emptyset, L(y,\ell+1)\in Q,
    \ell \in I_s, B_{y,s-1}, P_t(x) = Q} \\
&= \sum_{Q\in\F} \sum_{\ell=0}^{t-1} \ind{B_{x,t}}(Q) \sum_{\substack{u \geq n \beta^{s} \\ u \notin Q}}
    \pr{P_{\ell}(y) \cap Q = \emptyset, L(y,\ell) = u, \floor{uX_u} \in Q, \ell \in I_s, B_{y,s-1}, P_t(x) = Q}.
\end{align*}

In order to simplify this expression, we use the independence
claim below.
\begin{claim}
For any $Q \subseteq \{0, \dots, 2n\}$, $u \notin Q$ and $\ell \in \N$, the 
events $\event{\floor{uX_u} \in Q}$, $\event{P_t(x) = Q}$ and 
$E\eqdef \event{P_{\ell}(y) \cap Q = \emptyset, L(y,\ell) = u, \ell \in I_s, B_{y,s-1}}$ 
are mutually independent.
\end{claim}
\begin{proof}
As in Lemma \ref{lem:jointproba}, the event  $\event{\floor{uX_u}
\in Q}$ is in the sigma-algebra generated by $X_u$ and
$\event{P_t(x) = Q}$ is in the sigma-algebra generated by $\{X_w:
w \in Q\}$. So we only show that $E$ is in the sigma-algebra
generated by $\{X_w: w \notin Q, w \neq u\}$.

By looking just at variables from $\{X_w: w \notin Q, w \neq u\}$,
it is possible to determine the path of length $\ell$ starting at
$y$ until it reaches a node in $Q \cup \{u\}$. If any node in $Q$
is reached before $\ell$ steps, then $\event{P_{\ell}(y) \cap Q =
\emptyset}$ cannot hold. Moreover, if node $u$ is reached before
$\ell$ steps, $\event{L(y,\ell)=u}$ cannot hold. Otherwise,
knowing the path $P_{\ell}(y)$, it is easy to determine whether
$\ell \in I_s$. If in fact $\ell \in I_s$, we know that
$T(y,s-1)\leq \ell$. So either $\ell \geq s-1$ in which case we
can clearly determine if $B_{y,s-1}$ holds, or $\ell < s-1$ but
then rewriting $B_{y,s-1}$ as
\[
B_{y,s-1}=[T(i,1)\leq 1, T(i,2)\leq 2, \dots, T(i,s-1)\leq s-1],
\]
we can see that it is possible to determine whether $B_{y,s-1}$
holds or not.
\end{proof}

It follows that
\begin{align*}
&\pr{ T \in I_s, B_{x,t}\cap B_{y,t} } \\
&\leq \sum_{Q\in\F} \sum_{\ell=0}^{t-1} \ind{B_{x,t}}(Q) \sum_{u: \;\substack{u \geq n \beta^{s} \\ u \notin Q}} \pr{E} \pr{P_t(x) = Q} \pr{\floor{uX_u} \in Q} \\
&\leq \sum_{\ell=0}^{t-1}  \left(\sum_{Q\in\F} \ind{B_{x,t}}(Q) \pr{P_t(x) = Q} \right) \pr{B_{y,s-1}} (t+1)
    \sup_{\substack{u:\; u \geq n \beta^{s} \\ w:\; w < +\infty}} \pr{\floor{uX_u} = w} \\
&= \pr{B_{x,t}} \pr{B_{y,s-1}} t(t+1)
    \sup_{\substack{u:\; u \geq n \beta^{s} \\ w:\; w < +\infty}} \pr{\floor{uX_u} = w}.
\end{align*}
We can assume that $Q$ does not contain the node $+\infty$ because
otherwise $B_{x,t}$ does not hold. Thus we can use the bound
$\kappa$ on the density to get
\begin{align*}
\pr{T \in I_s, B_{x,t}\cap B_{y,t}} &\leq \pr{B_{x,t}} \pr{B_{y,s-1}} \frac{t(t+1)\kappa}{n \beta^{s-1}}.
\end{align*}
Observing that the above argument can be repeated for $T\in
I_{\infty}$, we get
\[
\pr{T\in I_{\infty}, B_{x,t}\cap B_{y,t}}  \leq \pr{B_{x,t}} \pr{B_{y,t}}. \qedhere
\]
\end{proof}

We omit the proof of the next lemma as it is similar to the proof
of Lemma \ref{lem:unbounded}.
\begin{lemma}\label{lem:unboundedmin}
Assume that $X \in [0,1)$ has a density and $\ex{-\log X} < +\infty$, and let $z \leq -\mu$ be
such that $\leg{z} < +\infty$. Then for all $\delta > 0$, there
exists $X_{\delta} \geq X$ such that $\law(X_{\delta})$ has a
bounded density and an atom at $+\infty$, such that $\ex{\log X_{\delta}}$ is well-defined and
\[
\Lambda^*(z) \leq \Lambda_{\delta}^*(z) \leq \Lambda^*(z) + \delta.
\]
\end{lemma}

We can now prove the main theorem of this section.
\theoremstyle{theorem}
\newtheorem*{theorem-mindepth}{Theorem \ref{thm:mindepthmain}}
\begin{theorem-mindepth}[Restated]
The minimum depth $M_n$ of a \textsc{sarrt} with attachment $X$
having a density, bounded or not, satisfies
\[
\frac{M_n}{\log n} \inprob \alpha_{\min} \quad\text{as}\quad n\to\infty,
\]
where $\alpha_{\min}$ is defined in equation \eqref{eq:alphamin}.
\end{theorem-mindepth}
\begin{proof}
Let $c \in (\alpha_{\min}, 1/\mu)$ and pick $\eps$ so that
$\eps/\mu < 1 - \Psi(c)$ (recall that $\mu = \ex{-\log X} > 0$ and that we can assume $\mu < + \infty$). In
order to handle the case where $X$ has an unbounded density, we
define (using Lemma \ref{lem:unboundedmin}) an auxiliary random variable $X_{\eps}\geq X$ with an atom
at $+\infty$ and a density on $(0,1)$ bounded by $\kappa =
\kappa(\eps)$ such that for all $z \leq -\mu$ such that
$\Lambda^*(z) < +\infty$, we have
\[
\Lambda^*(z) \leq \Lambda^*_{\eps}(z) \leq \Lambda^*(z) + \eps.
\]
Define $\Psi_{\eps}(c) = c \Lambda^*_{\eps}(-1/c)$ and
$\wt{\alpha}_{\min} = \sup \left\{\, 0 \,\right\} \cup \left\{c:\;
c \in \R_+ \;\text{ and }\; \Psi_{\eps}(c) > 1\right\}$.  By the
choice of $c$ and $\eps$,
\[
\Psi(c) \leq \Psi_{\eps}(c) \leq \Psi(c) + c\eps < \Psi(c) + \eps/\mu < 1
\]
so that $c > \wt{\alpha}_{\min}$.

Consider a sequence of independent random variables $\wt{X}_0,
\dots, \wt{X}_{2n}$ distributed as $X_{\eps}$, constructed as in
Lemma \ref{lem:unboundedmin} so that $X_i \leq \wt{X}_i$ for all
$1\leq i\leq 2n$. We can define the associated ancestor labels
$\wt{L}(x,s)$ and events $\wt{B}_{x,s}$ for any $x \in \{0, \dots,
2n\}$ and $s \geq 1$. Because $X_i \leq \wt{X}_i$ for every $1\leq
i\leq 2n$ we have for all $t \geq 1$ and $\beta \in (0,1)$,
\[
\pr{\bigcup_{x=n+1}^{2n} B_{x,t}(\beta)} \geq \pr{\bigcup_{x=n+1}^{2n} \wt{B}_{x,t}(\beta)}.
\]
To prove that $\pr{\cup_{x=n+1}^{2n} \wt{B}_{x,t}(\beta)}$
approaches $1$ as $n \to \infty$, we proceed in a similar way as
in Theorem \ref{thm:height}. Fix $\delta \in (0,1/2)$ with
$3\delta < \eps$, $\beta = e^{-1/c}$ and $t = \floor{(1-\eps) c
\log n}$. We have
\begin{equation}\label{eq:secondmomentmin}
\pr{\bigcup_{x=n+1}^{2n} \wt{B}_{x,t}} \geq \frac{\left(\sum_{x=n+1}^{2n} \pr{\wt{B}_{x,t}}\right)^2 }
{\sum_{x=n+1}^{2n} \pr{\wt{B}_{x,t}} + \sum_{x \neq y} \pr{\wt{B}_{x,t} \cap \wt{B}_{y,t}}}.
\end{equation}
First, as $c < \wt{\alpha}_{\min}$, we can use Lemma \ref{lem:probb}:
\[
\pr{\wt{B}_{x,t}} \geq \frac{\beta^t}{t} \geq \frac{n^{-1+\eps}}{t}.
\]
Then, using Lemma \ref{lem:jointprobb}, we get
\[
\pr{\wt{B}_{x,t} \cap \wt{B}_{y,t}} \leq \sum_{s=1}^t
\pr{\wt{B}_{x,t}} \pr{\wt{B}_{y,s-1}} \frac{t(t+1) \kappa}{n \beta^{s-1}} + \pr{\wt{B}_{x,t}} \pr{\wt{B}_{y,t}}.
\]
Let $t_0$ be defined as in Lemma \ref{lem:probb}. A calculation
similar to the one in the proof of Theorem \ref{thm:height} gives:
\[
\pr{\wt{B}_{x,t}\cap \wt{B}_{y,t}} \leq \pr{\wt{B}_{x,t}}
\left( \Ord{\frac{t}{n}} +
\frac{t(t+1)\kappa}{n} \cdot \frac{\beta^{(\Psi(c)-\delta-1)t}-1}{\beta^{(\Psi(c)-\delta-1)}-1}
+ \pr{\wt{B}_{y,t}}\right).
\]

We end up with
\[
\pr{\wt{B}_{x,t} \cap \wt{B}_{y,t}} \leq \pr{\wt{B}_{x,t}} \pr{\wt{B}_{y,t}} \left(1 + \Ord{n^{-\eps/4}}\right).
\]
Thus, going back to equation \eqref{eq:secondmomentmin}, we obtain
\[
\pr{\bigcup_{x=n+1}^{2n} B_{x,t}} \geq \pr{\bigcup_{x=n+1}^{2n} \wt{B}_{x,t}} \geq 1 - \Ord{n^{-\eps/4}}.
\]

When the event $B_{x,t}$ holds, $L(x,t) \leq 2n\beta^t \leq 2n
\cdot e^{1/c} n^{-1+\eps} \leq 2e^{1/c} n^{\eps}$, i.e., the
length of the path from $x$ to a node whose label is no larger
than $2e^{1/c} n^{\eps}$ is at most $t$. But using the upper bound
on the height of a \textsc{sarrt} (Section \ref{sec:heightup}), we
know that the depth of a node labeled $m$ is at most $2
\alpha_{\max} \log m$ with high probability (recall that
$\alpha_{\max} < +\infty$). In fact,
\begin{align*}
\mathbf{P}\{M_{2n} > c \log n &+ 2\eps \alpha_{\max} \log n\} \\
    &\leq \pr{M_{2n} > t + 2 \eps \alpha_{\max} \log n} \\
    &\leq \left(1-\pr{\bigcup_{x=n+1}^{2n} B_{x,t}}\right) + \pr{ \max_{1 \leq i \leq 2e^{1/c}n^{\eps}} D_i \geq 2 \eps \alpha_{\max} \log n} \\
    &\leq \Ord{n^{-\eps/3}} + \Ord{n^{1-\Psi(2\alpha_{\max})}}.
\end{align*}
We conclude that for any $\eps > 0$,
\[
\pr{M_{n} \leq (1+\eps)  \alpha_{\min} \log n } \to 1.
\]
Combining this with the upper bound proved in Lemma
\ref{lem:lower}, we get the desired result.
\end{proof}


\section{Applications}\label{sec:applications}

Giving $X$ the uniform $[0,1)$ density provides a new elementary
proof for the height of the \textsc{urrt} that avoids any mention
of branching processes as has been done by \citet{Dev87} or
\citet{Pit94}. Note that Cram\'{e}r's Theorem is not needed in this
case. Instead,  Proposition \ref{prop:cramer} can be directly
proven in this case using properties of the gamma distribution.

Moreover, setting $X = \max (U_1, \dots, U_k)$ and $X = \min(U_1,
\dots, U_k)$, we can compute asymptotics for greedy distances
introduced in \citet{DJ09}. A random $k$-\textsc{dag} (or
\textsc{urrt}) is a directed graph defined as follows. For each
node $i = 1, \dots, n$, a random set of $k$ parents is picked with
replacement uniformly from among the previous nodes $\{0, \dots,
i-1\}$ and the root is still labeled $0$. A node of the graph has
many paths going to the root. One can define many distances. Some
aspects of the longest path distance were studied in \citet{AGM99,TX96} 
and the shortest path distance in \citet{DJ09}.
Moreover, the authors of \cite{DJ09} introduced two other distances defined by
picking the path to the root following the smallest or largest
labels. For instance, if one chooses the parent with the smallest
label, this label is distributed as $\min(\floor{nU_1},
\floor{nU_2}, \dots, \floor{nU_k}) = \floor{n\min(U_1, \dots,
U_k)}$. As a result, these distances can be studied in the
framework introduced in this paper. We define $R_i^-$ and $R_i^+$
to be the distance from node $i$ to the root following these
minimum and maximum label paths. These distances can also be seen
as the depths of node $i$ in a \textsc{urrt} where each node is
given a choice of $k$ independent parents.  The random variable
$R_i^{-}$ corresponds to the choice of the parent with the
smallest label (oldest node) and $R_i^+$ corresponds to the choice
of the newest parent.

Let $X_{\max} = \max (U_1, \dots, U_k)$. Then, by Theorems
\ref{thm:depth}, \ref{thm:heightmain} and \ref{thm:mindepthmain},
\[
\frac{R_n^+}{\log n} \inprob \rho^+ = k \qquad\text{and}\qquad
\frac{R_n^+ - k\log n}{ \sqrt{k\log n} } \inlaw \mathcal{N}(0,1),
\]
and
\[
\lim_{n \to \infty} \frac{\max_{1 \leq i \leq n} R_i^+}{\log n} = \rho^+_{\max} 
\quad \text{almost surely,} \qquad\text{and}\qquad
\frac{\min_{n/2 \leq i \leq n} R_i^+}{\log n} \inprob \rho^+_{\min},
\]
where $\rho^+_{\min}$ and $\rho^+_{\min}$ are defined as the
solutions respectively smaller and larger than $k$ of the equation
$ -c + k - c \log \frac{k}{c} = 1$. Some numerical approximations
generated using a program are shown in Table \ref{tab:values}. It
should be noted that the concentration for $R_n^+$ as well as for
$R_n^-$ presented below were shown in \citet{DJ09} and
\citet{Mah09}, and the corresponding central limit theorems in
\citet{Mah09}.

We give expressions for the relevant functions introduced in the
proof:
\begin{align*}
\ex{-\log X_{\max}} &= \frac{1}{k},\\
\var{-\log X_{\max}} &= \frac{1}{k^2},\\
\Lambda(\lambda) &= - \log \left(1+\frac{\lambda}{k}\right),  & &\text{(for $\lambda>-k$)} \\
\Lambda^*(z) &= -1 - kz - \log(-kz),  & &\text{(for $z < 0$)} \\
\Psi(c) &= -c + k - c \log \frac{k}{c}.
\end{align*}

Similarily, let $X_{\min} = \min (U_1, \dots, U_k)$, then setting
$h_k = \sum_{i=1}^k \frac{1}{i}$ and $h^{(2)}_k = \sum_{i=1}^k
\frac{1}{i^2}$,
\[
\frac{R_n^-}{\log n} \inprob \rho^- = \frac{1}{h_k} \qquad\text{and}\qquad
\frac{R_n^- - \frac{\log n}{h_k}}{ \sqrt{\frac{h_k^{(2)}}{h_k^3} \log n} } \inlaw \mathcal{N}(0,1),
\]
and
\[
\lim_{n\to \infty} \frac{\max_{1 \leq i \leq n} R_i^-}{\log n} = \rho^-_{\max}
\quad \text{almost surely,}
\qquad\text{and}\qquad
\frac{\min_{n/2 \leq i \leq n} R_i^-}{\log n} \inprob \rho^-_{\min},
\]
where  $\rho^-_{\min}$ and $\rho^-_{\min}$ are defined as the
solutions respectively smaller and larger than $1/h_k$ of the
equation $\Psi(c) = 1$. See Table \ref{tab:values} for numerical
approximations of these constants for different values of $k$.

An expression for $\Psi$ and other relevant functions are given
for $X_{\min}$:
\begin{align*}
\ex {-\log X_{\min}} &= h_k,\\
\var {-\log X_{\min}} &= h^{(2)}_k,\\
\Lambda(\lambda) &= - \sum_{i=1}^k \log \left(1+\frac{\lambda}{i}\right),
    & &\text{(for $\lambda>-k$)} \\
\Lambda^*(z) &=  \lambda^*_k(z) z + \sum_{i=1}^k \log\left(1 + \frac{\lambda^*_k(z)}{i}\right),
    & &\text{(for $z < 0$)} \\
\Psi(c) &= - \lambda^*_k(-1/c) + \sum_{i=1}^k \log\left(1+\frac{\lambda^*_k(-1/c)}{i}\right),
\end{align*}
where $\lambda_k^*(z)$ is the solution of $z +\sum_{i=1}^{k} \frac{1}{1+\lambda^*_k(z)/i} = 0$.

\begin{table}[h]
\caption{Approximate numerical values for some constants}
\label{tab:values}
\begin{tabular}{ | c | c c c | c c c |}
  \hline
  $k$ & $\rho_{\min}^+$ & $\rho^+$ & $\rho_{\max}^+$
         & $\rho_{\min}^-$ & $\rho^-$ & $\rho_{\max}^-$ \\
  \hline
  1 & 0 & 1 & $e$ & 0 & 1 & $e$ \\
  2 & 0.3734 & 2 & 4.3111 & 0 & 0.6667 & 1.6738 \\
  3 & 0.9137 & 3 & 5.7640 & 0 & 0.5455 & 1.3025 \\
  4 & 1.5296 & 4 & 7.1451 & 0 & 0.4800 & 1.1060 \\
  5 & 2.1925 & 5 & 8.4805 & 0 & 0.4380 & 0.9818 \\
  \hline
  \end{tabular}
\end{table}

\begin{remark}
This of course can be repeated for $k$-\textsc{dag}s where the
parents of node $n$ are independent and distributed as
$\floor{nX}$ where $X \in [0,1)$ (\textsc{sarrd}) and $\law(X)$
has any density.
\end{remark}


\section{Concluding remarks}

To compute the height of the tree, our proof uses the existence of
a density for $\law(X)$ in order to bound the collision
probability. The existence of a density is only used to find a
lower bound on the height. The upper bound given here (Lemma
\ref{lem:upper}) works for any distribution. It is natural to ask
whether this upper bound is tight for a larger family of
distributions, for example when $\law(X)$ has atoms. Atoms at $0$
are handled by our proof. Note that for a deterministic $X =
\theta \in (0,1)$, the height of the tree, which is simply the
depth of node $n$, is $(1+o(1)) \frac{\log n}{\log 1/\theta}$. For
example, if $\theta = \frac{1}{m}$ for an integer $m \geq 2$, the
tree is a complete $m$-ary tree.

One can construct a random $k$-\textsc{dag} or \textsc{sarrd} in
the same way. Node $n$ chooses $k$ parents $\floor{nX^{(1)}},
\floor{nX^{(2)}}, \dots, \floor{nX^{(k)}}$ where $X^{(1)}, \dots,
X^{(k)}$ are independent copies of a random variable $X \in
[0,1)$. The ``greedy'' distance measures can be computed simply by
considering the \textsc{sarrt} with attachment random variable
$X_{\min} = \min(X^{(1)}, \dots, X^{(k)})$ and $X_{\max} =
\max(X^{(1)}, \dots, X^{(k)})$. One could study the shortest and
longest path distances in a \textsc{sarrd}, which has been done
for the uniform case in \citet{AGM99,DJ09,DKM07},
and \citet{TX96}.

Another point mentioned in \citet{DJ09} is the relation between the
\textsc{sarrt} model and random binary search trees
(\textsc{rbst}). A \textsc{rbst} can be constructed incrementally
by choosing one of the $n+1$ external node at random and replacing
it by the node that arrives at time $n$. The (random) arrival time
of the parent of $n$ is roughly distributed as $\max(\floor{U_1n},
\floor{U_2n})$. This suggests that the depth of nodes in a
\textsc{rbst} and in a \textsc{sarrt} with attachment $X =
\max(U_1, U_2)$ are related. Observe that the height of these two
different types of random trees are the same up to lower order
terms: $\frac{H_n}{\log n} \to \alpha$ where $\alpha \approx
4.3111$ \cite{Dev86}. Considering a best-of-two-choices
\textsc{rbst} in which each new node $n$ has two choices of keys,
and chooses the one for which the parent arrived last. It would be
interesting if the first order of the asymptotic height is the
same for a best-of-two-choices \textsc{rbst} and for an
\textsc{sarrt} with $X = \min(\max(U_1,U_2), \max(U_3, U_4))
\eqlaw \sqrt{1 - \sqrt{U}}$ whose limit $\frac{H_n}{\log n}
\inprob c$ where $c \approx 2.364$. If one picks the parent
closest to the root, then the analysis seems to be even more
challenging.

\section*{Acknowledgments}
The authors would like to thank the referees for their valuable comments.


\appendix
\section{Some pictures of \textsc{sarrt}s}
We include some pictures of \textsc{sarrt} for attachment random
variable of the form $U^{\beta}$ for different values of $\beta$
where $U$ is uniform in $[0,1)$. We color the nodes from light
(green) to dark (red) as a linear function of their labels.

Note that for small values of $\beta$, the attachment distribution
concentrates more around $1$ and most of the nodes link to nodes
of labels close to the bottom part of the tree. As $\beta$ becomes
larger, the distribution is more concentrated near $0$. The tree
has a smaller height, and the root's degree increases.

\begin{figure}[h]
\centering
\subfloat[$\beta=1/2$]{\includegraphics[width=7.5cm]{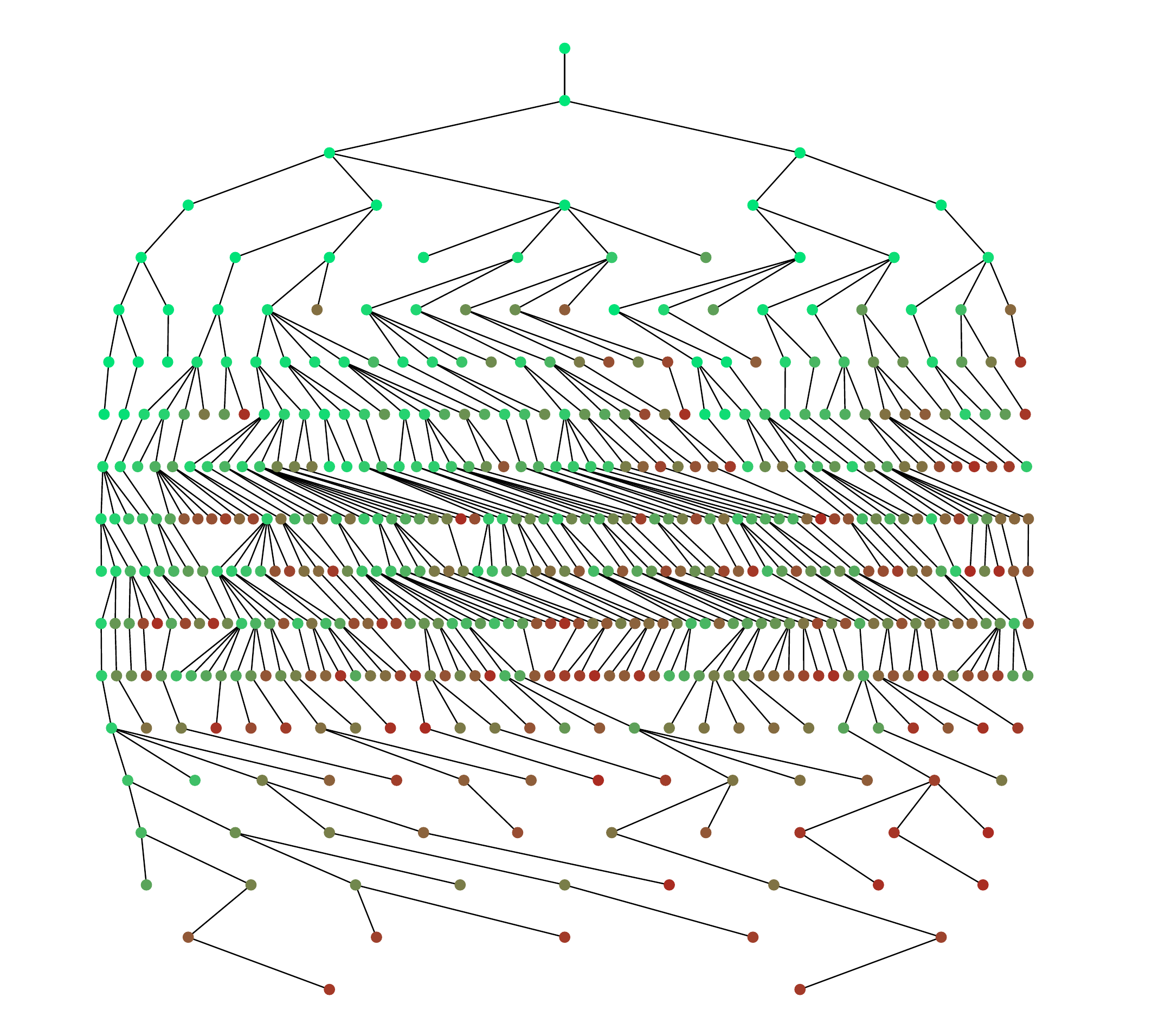}}\quad
\subfloat[$\beta=1$]{\includegraphics[width=7.5cm]{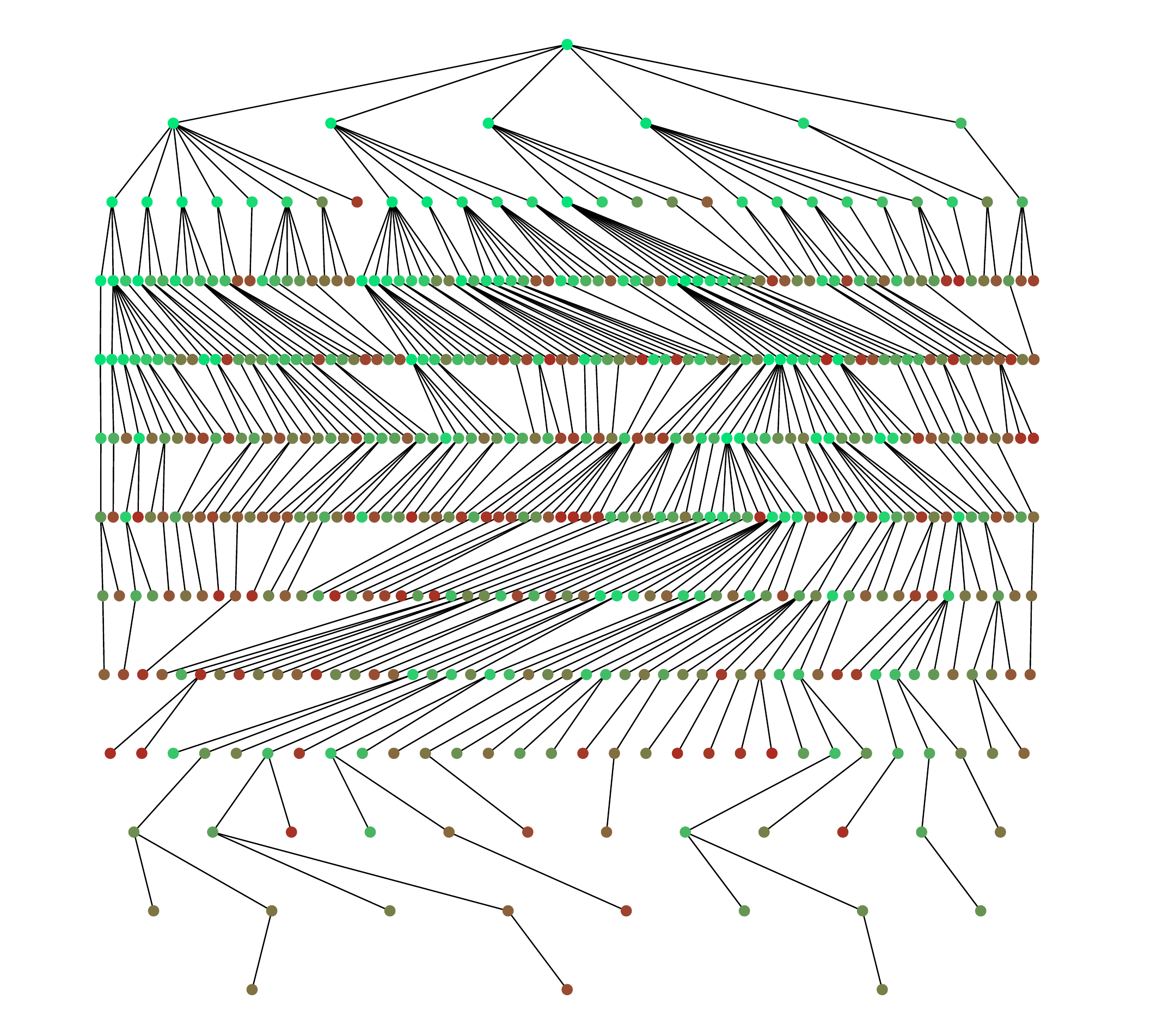}}\\
\subfloat[$\beta=2$]{\includegraphics[width=7.5cm]{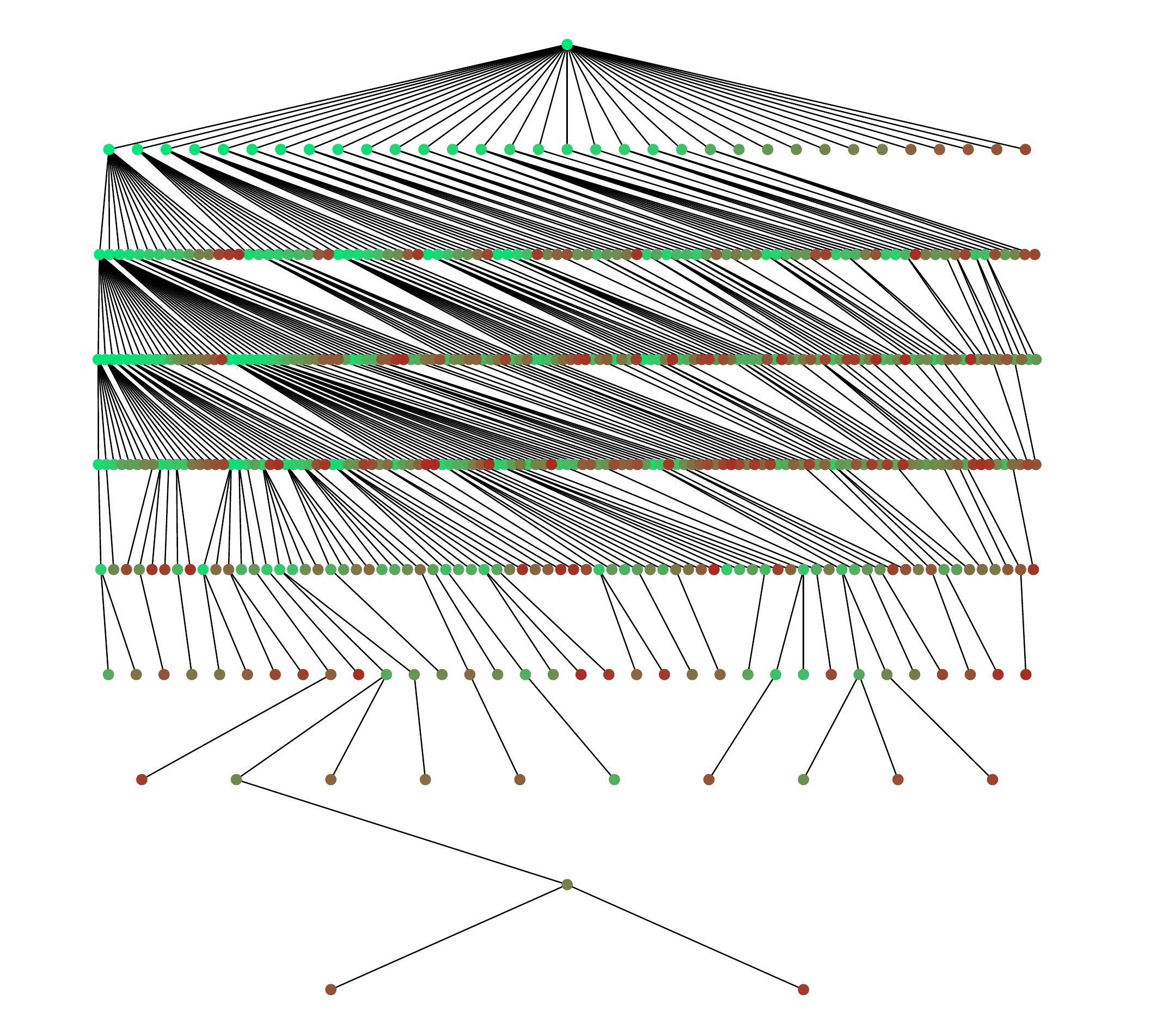}}\quad
\subfloat[$\beta=3$]{\includegraphics[width=7.5cm]{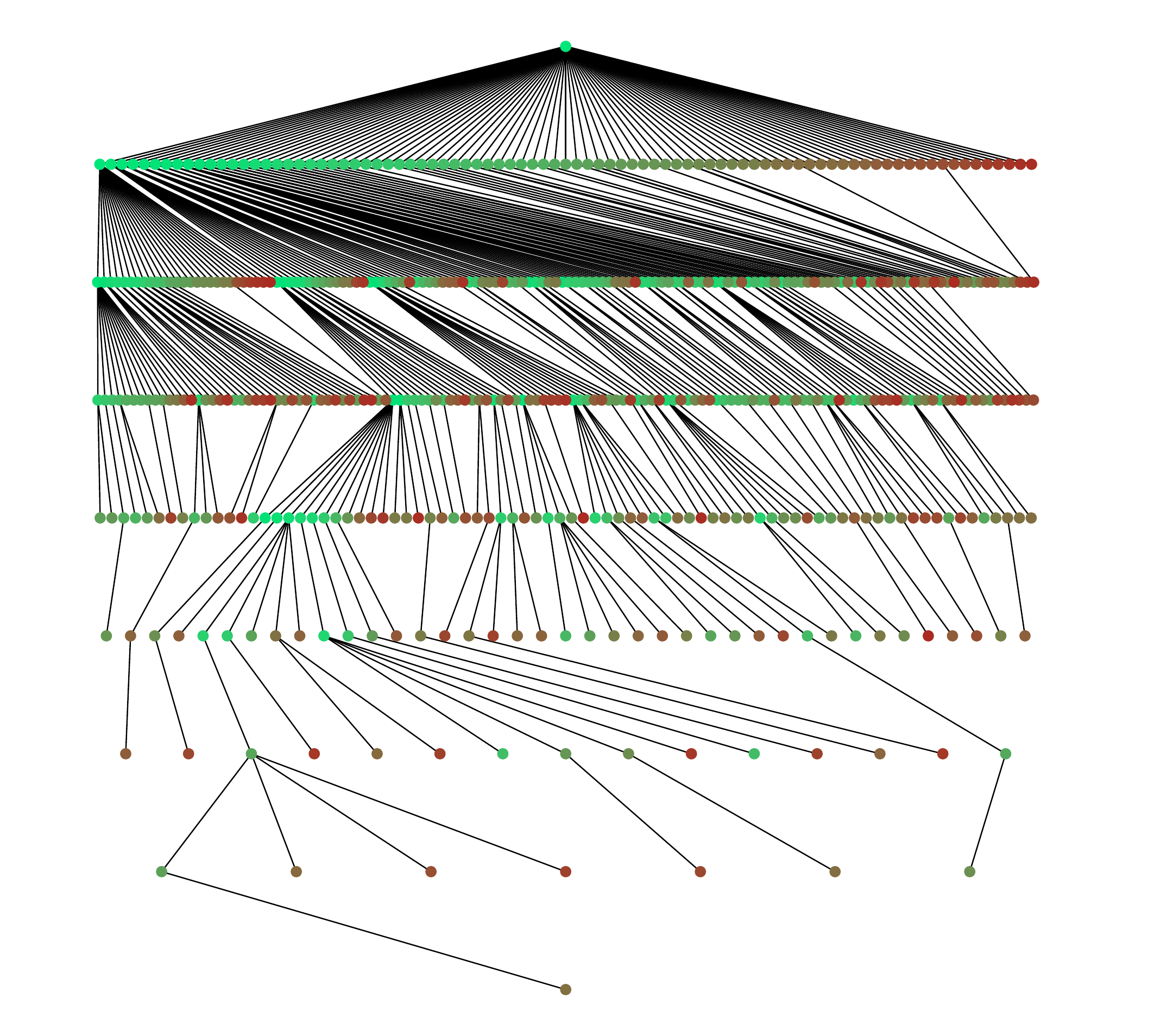}}
\caption{\textsc{sarrt} with distribution $U^{\beta}$ and $n= 500$.}
\end{figure}


\section{Properties of $\Lambda^*$ and $\Psi$}\label{sec:applambda}

We prove some properties of Cram\'er's function $\Lambda^*\,$ as
well as the function $\Psi$ both defined in Section
\ref{sec:heightup}. See \cite{DZ98} for more details on Cram\'er's
theorem. Recall that $\Lambda^*$ is defined as:
\[
\Lambda^*(z) = \sup_{\lambda\in\R} \big\{\lambda z - \Lambda(\lambda) \big\},
\qquad\text{where}\quad
\Lambda(\lambda) = \log \ex{e^{\lambda Y}}.
\]
Note that in our case $Y = \log X$ is a negative random variable,
so $\Lambda(\lambda) < +\infty$ for $\lambda \geq 0$.

\begin{proposition}\label{prop:lambdastar}
Let $Y$ be a negative random variable with $\ex{Y} = -\mu \in [-\infty, 0)$. Then:
\begin{enumerate}[\textup{(}i\textup{)}]
\item $\leg{z} \in [0, +\infty]$ for all $z \in \R$.
\item If $\mu < +\infty$, then $-\mu \in D_{\Lambda^*}$ and $\Lambda^*(-\mu) = 0$.
\item $\leg{z} = \sup_{\lambda \geq 0} \big\{\lambda z - \Lambda(\lambda) \big\} \quad\text{if } z \geq -\mu$.
\item If $\Lambda(\lambda) < +\infty$ for some $\lambda < 0$, $\leg{z} = \sup_{\lambda \leq 0} \big\{\lambda z - \Lambda(\lambda) \big\} \quad\text{for } z \leq -\mu$.
\item $\Lambda^*$ is decreasing on $(-\infty, -\mu)$ and increasing on $(-\mu, +\infty)$.
\item $\leg{z} > 0$ for $z > -\mu$.
\item $\Lambda^*$ is convex and thus continuous on the interior of $\{ z: \leg{z} < +\infty \}$.
\end{enumerate}
\end{proposition}
\begin{proof}\mbox{}
\begin{enumerate}[(i)]
\item $\leg{z}$ is non-negative for $z \in \R$:
\[
\leg{z} \geq 0 \cdot z - \Lambda(0) = 0.
\]

\item By concavity of the logarithm function, we have
\begin{equation}\label{eq:lambda_ineq}
\Lambda(\lambda) = \log \ex{e^{\lambda Y}} \geq \ex{ \log e^{\lambda Y} } = \lambda \ex{Y} = -\lambda \mu,
\end{equation}
using Jensen's inequality. As a result
\[
\leg{\mu} = \sup_{\lambda} \big\{\lambda \mu - \Lambda(\lambda) \big\} \leq 0.
\]
We conclude using the non-negativity of $\Lambda^*$.

\item If $\Lambda(\lambda) < +\infty$ for some $\lambda < 0$, then
$\mu < +\infty$. In fact, $\Lambda(\lambda) < +\infty$ implies
\[
\ex{Y} \geq \ex{ e^{-\lambda Y}}/\lambda > -\infty
\]
by using the inequality $\lambda z \leq e^{\lambda z}$ for all
reals $\lambda$ and $z$. It follows that if $\mu = +\infty$,
$\Lambda(\lambda) = +\infty$ for all $\lambda <0$. In this case,
the property trivially holds. For $\mu$ finite, $z \geq -\mu$ and
$\lambda \leq 0$
\[
\lambda z - \Lambda(\lambda) \leq -\lambda\mu - \Lambda(\lambda)\leq \leg{-\mu} = 0.
\]

\item As previously shown, we have $\mu < +\infty$ in this case.
For $z \leq -\mu$ and $\lambda \geq 0$,
\[
\lambda z - \Lambda(\lambda) \leq -\lambda \mu - \Lambda(\lambda) \leq \leg{-\mu} = 0.
\]
\item For $z \geq -\mu$, $\leg{z} = \sup_{\lambda \geq 0} \big\{
\lambda z - \Lambda(\lambda) \big\}$. This implies that
$\Lambda^*$ is increasing on $[-\mu, +\infty)$ as $z \mapsto
\lambda z - \Lambda(\lambda)$ is an increasing function. Now if
$\Lambda(\lambda) < +\infty$ for some $\lambda < 0$, then $\leg{z}
= \sup_{\lambda \leq 0} \big\{\lambda z - \Lambda(\lambda) \big\}$
for $z \leq -\mu$, and similarly we get $\Lambda^*$ decreasing on
$(-\infty, -\mu)$. Otherwise if $\Lambda(\lambda) = +\infty$ for
all $\lambda < 0$, then $\Lambda^*(z) = 0$ for all $z \leq -\mu$.

\item For $z > -\mu$, consider the function $f : \lambda \mapsto
\lambda z - \Lambda(\lambda)$. As $Y$ is a negative random
variable, this function is defined for all $\lambda \geq 0$.
Moreover it is differentiable and $f'(\lambda) = x -
\ex{Ye^{\lambda Y}}/\ex{e^{\lambda Y}}$. Observe that $f'(0) = z -
\mu > 0$. As a result $f$ is positive on a neighborhood of $0$. As
a result $\leg{z} = \sup_{\lambda} \{ f(\lambda) \} > 0$ on this
interval. Now as $\Lambda^*$ is increasing, we get the desired
result.

\item For $\theta \in [0,1]$,
\begin{align*}
\theta \leg{z_1} + (1-\theta) \leg{z_2}
    &= \sup_{\lambda\in\R} \big\{ \theta \lambda z_1 - \theta \Lambda(\lambda) \big\}
    + \sup_{\lambda\in\R} \big\{ (1-\theta) \lambda z_2 - (1-\theta) \Lambda(\lambda) \big\} \\
    & \geq \sup_{\lambda\in\R} \big\{  \lambda (\theta z_1 + (1-\theta) z_2) - \Lambda(\lambda) \big\} \\
    &= \leg{\theta z_1 + (1-\theta) z_2}. \qedhere
\end{align*}
\end{enumerate}
\end{proof}

In the next proposition, another property of $\Lambda^*$ is
introduced to prove that except in the case where $\law(Y)$ is a
single mass, there exists $z > -\mu$ for which $\Lambda^*$ is
finite.

\begin{proposition}\label{prop:lambdastar_finite}
Let $Y$ be a negative random variable with $\ex{Y} = -\mu \in
[-\infty, 0)$. If $\law(Y)$ is not a single mass, then there
exists $z > -\mu$ such that $\Lambda^*(z) < +\infty$. Moreover if
$\Lambda(\lambda) < +\infty$ for some $\lambda < 0$, then there
exists also $z < -\mu$ such that $\Lambda^*(z) < +\infty$.
\end{proposition}
\begin{proof}
We start by proving that $\Lambda(\lambda)/\lambda$ is a strictly
increasing function for $\lambda > 0$. Writing $X = e^{Y}$, we
have $\log \ex{ e^{\lambda Y} } = \log \ex{ X^{\lambda} }$. Let $0
< \lambda < \lambda'$, and define $g(x) = x^{\lambda'/\lambda}$
for $x \geq 0$. Then using Jensen's inequality for the convex
function $g$:
\[
\ex{ X^{\lambda} }^{1/\lambda} = g\left(\ex{ X^{\lambda} }\right)^{1/\lambda'} < \ex{ g(X^{\lambda}) }^{1/\lambda'} = \ex{ X^{\lambda'} }^{1/\lambda'}.
\]
as $X$ is not constant. By taking the logarithm
\[
\frac{\Lambda(\lambda)}{\lambda} < \frac{\Lambda(\lambda')}{\lambda'}.
\]
Let $z_1 = \Lambda(1)$. By the fact that
$\Lambda(\lambda)/\lambda$ is increasing, $\lambda (z_1 -
\Lambda(\lambda)/\lambda) \leq 0$ for $\lambda \geq 1$. Thus,
\[
\leg{z_1} = \sup_{\lambda \geq 0} \big\{ \lambda (z_1 - \Lambda(\lambda)/\lambda) \big\} =
\sup_{0 \leq \lambda \leq 1} \big\{ \lambda z_1 - \Lambda(\lambda) \big\} < +\infty.
\]
Now, using equation \eqref{eq:lambda_ineq}, $\Lambda(0.5)/0.5 \geq
-\mu$. But $z_1 = \Lambda(1)/1 > \Lambda(0.5)/0.5 \geq -\mu$.
Finally, $z_1 > -\mu$ and $\Lambda^*(z_1) < +\infty$.

As for the case $z < -\mu$, we start by observing that
$\Lambda(\lambda)/\lambda$ is a strictly decreasing function of
$\lambda$ for $\lambda < 0$ using the same argument as above. Then
if $\Lambda(\delta) < +\infty$ for some $\delta < 0$, let
$z_{\delta} = \Lambda(\delta)/\delta$. We have $z_{\delta} >
\Lambda(0.5\delta)/0.5 \delta \geq -\mu$. Moreover, $\lambda
(z_{\delta} - \Lambda(\lambda)/\lambda) \leq 0$ for $\lambda \leq
\delta$. Thus,
\[
\leg{z_{\delta}} = \sup_{\delta \leq \lambda \leq 0} \big\{ \lambda (z_{\delta} - \Lambda(\lambda)/\lambda) \big\} < +\infty. \qedhere
\]
\end{proof}

Using these properties we prove the results needed for the
function $\Psi$.

\begin{proposition}\label{prop:psi}
Let $Y$ be a negative random variable with $\ex{Y} = -\mu \in
[-\infty, 0)$. Define the function $\Psi$ by $\Psi(c) = c
\leg{-1/c}$ for $c > 0$. Let $\mathcal{D}_{\Psi} = \{c > 0:
\Psi(c) < +\infty \}$. Then,
\begin{enumerate}[\textup{(}i\textup{)}]
\item The function $\Psi$ is continuous on the interior of
$\mathcal{D}_{\Psi}$. It is decreasing on $(0, 1/\mu)$ and
strictly increasing on $(1/\mu, +\infty) \cap \mathcal{D}_{\Psi}$.

\item The set $\{ c > 1/\mu : \Psi(c) > 1\}$ is non-empty. Define
\[
\alpha_{\max} = \inf \left\{ c > \frac{1}{\mu} : \Psi(c) > 1 \right\}.
\]
Then $\alpha_{\max} < +\infty$, and if $\law(X)$ is not a single
mass, $\alpha_{\max} > 1/\mu$. Moreover, for $c \in (1/\mu,
\alpha_{\max})$, then $\Psi(c) < 1$.

\item If $\mu < +\infty$, define
\[
\alpha_{\min} = \sup \bigg\{\,0\,\bigg\} \cup \left\{ c < \frac{1}{\mu} : \Psi(c) > 1 \right\}.
\]
Then if $\law(X)$ is not a single mass, $\alpha_{\min} < 1/\mu$.
Moreover, for $c \in (\alpha_{\min},1/\mu)$, we have $\Psi(c) <
1$.
\end{enumerate}
\end{proposition}
\begin{proof}\mbox{}
\begin{enumerate}[(i)]
\item The continuity follows from the continuity of $\Lambda^*$.
For $(1/\mu, +\infty) \cap \mathcal{D}_{\Psi}$, $\Psi$ is strictly
increasing because $\Lambda^*$ is increasing and $\leg{z} > 0$ for
$z > -\mu$. For $(0, 1/\mu) \cap \mathcal{D}_{\Psi}$, using the
convexity of $\Lambda^*$, we have for $z < z' \leq -\mu$ in
$\mathcal{D}_{\Psi}$:
\[
\frac{\leg{z}}{-z} \geq \frac{\leg{z'}}{-z'}.
\]
Thus, $\Psi(-1/z) \geq \Psi(-1/z')$ and $\Psi$ is decreasing on
$(0, 1/\mu) \cap \mathcal{D}_{\Psi}$.

\item Fix any $z' \in (-\mu, 0)$, then using the positivity of
$\Lambda^*$, $\leg{z'} > 0$ and thus for $c \geq -1/z'$,
\[
\Psi(c) = c \Lambda^*(-1/c) \geq c \Lambda^*(z').
\]
As a result, for $c$ large enough $\Psi(c) > 1$. This shows that
$\alpha_{\max} < +\infty$. Moreover, if $\law(X)$ is not a single
mass, then Proposition \ref{prop:lambdastar_finite} and the
continuity of $\Lambda^*$ imply that $\Psi$ is smaller than $1$ on
an interval $[1/\mu, c]$ for some $c > 1/\mu$. This shows that
$\alpha_{\max} > 1/\mu$.

Furthermore, taking $c < \alpha_{\max}$, by definition of
$\alpha_{\max}$ and as $\Psi$ is strictly increasing on $(1/\mu,
+\infty)$, $\Psi(c) < 1$.

\item First, if $\Lambda(\lambda) = +\infty$ for all $\lambda < 0$,
then $\leg{z} = 0$ for all $z<-\mu$. In this case, $\alpha_{\min}
= 0 < 1/\mu$ and $\Psi(c) = 0 < 1$ for all $c \in (\alpha_{\min},
1/\mu)$.

Assume now that $\Lambda(\lambda) < +\infty$ for some $\lambda <
0$. Then using Proposition \ref{prop:lambdastar_finite}, we have
$\alpha_{\min} < 1/\mu$. It remains to show that for $c \in
(\alpha_{\min}, 1/\mu)$, $\Psi(c) < 1$. Suppose for the sake of
contradiction that this is not the case. Then there exists $c >
\alpha_{\min}$ such that $\Psi(c) = 1$. As $\Psi$ is a decreasing
function in $(\alpha_{\min}, 1/\mu)$, this implies that there
exists $z'_1 < z'_2 < 1/\mu$ such that $\leg{z} = -z$ for all $z
\in [z'_1,z'_2]$. But for $z_1 < z_2$ in $(z'_1, z'_2)$, we have
\begin{align}\label{eq:conv}
\leg{\frac{z_1 + z_2}{2}} &= \sup_{\lambda \leq 0} \left\{ \lambda \frac{z_1+z_2}{2} - \Lambda(\lambda) \right\} \notag \\
&\leq \frac{1}{2} \sup_{\lambda \leq 0} \big\{ \lambda z_1 - \Lambda(\lambda) \big\}  + \frac{1}{2} \sup_{\lambda \leq 0} \big\{ \lambda z_2 - \Lambda(\lambda) \big\} \\
                &= -\frac{z_1+z_2}{2}. \notag
\end{align}
So we must have equality in \eqref{eq:conv}. This means that the
suprema defining $\Lambda^*(z_1)$ and $\Lambda^*(z_2)$ are
attained at the same point. We have $\Lambda^*(z_1) = \lambda z_1
- \Lambda(\lambda)$ and $\Lambda^*(z_2) = \lambda z_2 -
\Lambda(\lambda)$ for some $\lambda < 0$. Observing that
$\Lambda^*(z_1) - \Lambda^*(z_2) = \lambda (z_1 - z_2)$, we must
have $\lambda = -1$. This implies that $\Lambda^*(z_1) = -z_1 -
\Lambda(-1) = -z_1$. But $\Lambda(-1) = \log \ex{X^{-1}} > 0$.
This contradicts our assumption that $\Psi(c) = 1$ for some $c >
\alpha_{\min}$. Note that we supposed here that for $z \in \{z_1,
z_2\}$ there exists some $\lambda$ such that $\leg{z} = \lambda z
- \Lambda(\lambda)$. In the next paragraph, we show that we can
suppose this is the case.

Fix some $z \in [z_1, z_2]$. We want to show that there exists a
$\lambda \leq 0$ such that $\leg{z} = \lambda z -
\Lambda(\lambda)$. Consider $\mathcal{D}_{\Lambda} = \{\lambda \in
\R : \Lambda(\lambda) < +\infty \}$ and let $a = \inf
\mathcal{D}_{\Lambda}$. Suppose first $a > -\infty$, and consider
the limit $\ell = \lim_{\lambda \downarrow a} \Lambda(\lambda)$.
This limit exists because $\Lambda$ is a decreasing function of
$\lambda$. If $\ell < +\infty$, then by extending $\Lambda$
by continuity, $\leg{z} = \sup_{a \leq \lambda \leq 0} \{\lambda z
- \Lambda(\lambda)\}$ so we can assume that the supremum is
attained. If $\ell = +\infty$, then there exists $a_1$ such that
$\Lambda(\lambda) \geq a z$ for $\lambda < a_1$. Thus, we have
$\leg{z} = \sup_{a_1 \leq \lambda \leq 0} \{\lambda z -
\Lambda(\lambda)\}$, and the supremum is also attained in this
case. Now suppose that $a = -\infty$ and define similarly $\ell =
\lim_{\lambda \to -\infty} \Lambda(\lambda)$. If $\ell < +\infty$,
then $\Lambda^*(z) = +\infty$ which is a contradiction. The last
case is $\ell = +\infty$. As $\Lambda$ is a convex function, the
function $\varphi : \lambda \mapsto \lambda z - \Lambda(\lambda)$
is a concave function so it is monotone for $\lambda \leq
\lambda_0$ small enough. If it is increasing, then $\leg{z} =
\sup_{\lambda_0 \leq \lambda \leq 0} \{\lambda z -
\Lambda(\lambda)\}$ and we are done. If $\varphi$ is decreasing
for $\lambda \leq \lambda_0$, then we can suppose $\leg{z} =
\lim_{\lambda \to -\infty} \lambda z - \Lambda(\lambda)$ and by assumption $\leg{z} = -z$. But
then for $z'_1< z$, we have $\Lambda^*(z'_1) \leq \lim_{\lambda \to
-\infty} \lambda(z'_1-z) + \lambda z - \Lambda(\lambda) = +\infty$,
which contradicts the fact that $\Lambda(z'_1) = -z'_1$.

\end{enumerate}
\end{proof}



\bibliographystyle{abbrvnat}
\bibliography{rrt}

\end{document}